\documentclass{amsart}

\usepackage{makecell}
\usepackage{placeins}
\usepackage{amssymb}
\usepackage{hyperref}

\newtheorem{proposition}{Proposition}[section]

\newtheorem{lemma}[proposition]{Lemma}

\newtheorem{theorem}[proposition]{Theorem}

\newcommand{\kpp}{\operatorname{kpp}}

\newcommand{\GL}{\operatorname{GL}}
\newcommand{\Out}{\operatorname{Out}}

\newcommand{\PSL}{\operatorname{PSL}}

\begin{document}

\title{More on Landau's theorem and conjugacy classes}

\author[B. \c{C}{\i}narc{\i}]{Burcu \c{C}{\i}narc{\i}}
\address[Burcu \c{C}{\i}narc{\i}]{Department of Mathematics, Texas State University, 601 University Drive, San Marcos, TX, 78666, USA}
\email{bcinarci@txstate.edu}

\author[T. M. Keller]{Thomas Michael Keller}
\address[Thomas Michael Keller]{Department of Mathematics, Texas State University, 601 University Drive, San Marcos, TX, 78666, USA}
\email{keller@txstate.edu}

\author[A. Mar\'oti]{Attila Mar\'oti}
\address[Attila Mar\'oti]{Hun-Ren Alfr\'ed R\'enyi Institute of Mathematics, Hungarian Academy of Sciences, Re\'altanoda utca 13-15, H-1053, Budapest, Hungary}
\email{maroti.attila@renyi.hu}

\author[I. I. Simion]{Iulian I. Simion}
\address[Iulian I. Simion]{Department of Mathematics\\
	`Babe\c s-Bolyai' University\\
	str. Ploie\c sti nr.23-25,
	400157 Cluj-Napoca\\
	Romania}
\email{iulian.simion@ubbcluj.ro}

\thanks{The third author was supported by the National Research, Development and Innovation Office (NKFIH) Grant No.~K138596, No.~K132951 and No.~K138828.}
\date{\today}

\begin{abstract}
In this paper we present two new results on the number of certain conjugacy classes of a finite group. For a finite group $G$, let $n(G)$ be the maximum of $k_{p}(G)$ taken over all primes $p$ where $k_{p}(G)$ denotes the number of conjugacy classes of nontrivial $p$-elements in $G$. Using a recent theorem of Giudici, Morgan and Praeger, we prove that there exists a function $f(x)$ with $f(x) \to \infty$ as $x \to \infty$ such that $n(G) \geq f(|G|)$ for any finite group $G$. Let $G$ be a finite group, and let $p$ be a prime dividing $|G|$. Let $k_{p'}(G)$ denote the number of conjugacy classes of elements of $G$ whose orders are coprime to $p$. We show that either $p=11$ and $G=C_{11}^2\rtimes \text{\rm SL}(2,5)$, or there exists a factorization $p-1 = ab$ with $a$ and $b$ positive integers, such that $k_{p}(G) \geq a$ and $k_{p'}(G) \geq b$ with equalities in both cases if and only if $G=C_p \rtimes C_b$ with $C_G(C_p) = C_p$.
\end{abstract}
\maketitle

\section{Introduction}

Let $G$ be a finite group. In the last two decades, there has been a lot of activity in establishing lower bounds for the number $k(G)$ of conjugacy classes of $G$. A classical problem is to bound $k(G)$ from below only in terms of $|G|$. The first result in this direction is due to Landau \cite{Landau} who showed that for any positive integer $k$ there are at most finitely many finite groups $G$ such that $k(G) = k$. Brauer \cite[p. 137]{Brauer} stated that Landau's proof can be used to show that, for $k(G)\geq 3$,
\[ |G|\leq (2k(G))^{2^{k(G)-3}}\ \prod\limits_{i=1}^{k(G)-2} (k(G)-i)^{2^{k(G)-2-i}}    \]
which leads to a bound of type $k(G)\geq c \log\log |G|$ for some constant $0<c<1$. (Here and
throughout the paper, the logarithms are taken to base $2$ unless otherwise stated.) The bound $k(G)\geq \log\log |G|$ was established in \cite[Corollary I]{ET}. Problem 3 of Brauer's list of problems \cite{Brauer} was to give a substantially better lower bound for $k(G)$ than this. This was achieved by Pyber in \cite{pyber1992}. His estimate was improved by the second author of this paper in \cite{Keller1}. The best general bound to date is due to Baumeister, the third author and Tong-Viet \cite{baumeister} and is of the order of magnitude $\log|G|/(\log\log |G|)^{3 + \epsilon}$, for any positive $\epsilon$.
Bertram \cite{Bertram} asks whether $k(G) > \log_{3}|G|$ holds for any finite group $G$.

Generalizing the theorem of Landau, Héthelyi and Külshammer \cite{hethelyi-kulshammer2005} proved that there exists a function $f$ on the set of natural numbers such that $\kpp(G) \geq f(|G|)$ for all finite groups $G$ and $f(x) \to \infty$ as $x \to \infty$. Here $\kpp(G)$ denotes the number of conjugacy classes of $G$ consisting of elements of prime power orders. For a nonabelian finite simple group $T$ and a prime $p$, let $m_{p}(T)$ be the number of orbits of the automorphism group $\mathrm{Aut}(T)$ of $T$ on the set of nontrivial $p$-elements of $T$. Let $m(T)$ be the maximum of $m_{p}(T)$ taken over all prime divisors of $|T|$. Recently, Giudici, Morgan and Praeger \cite{GMP1} proved the following surprising result: There exists an increasing function $f$ on the set of natural numbers such that, for a finite nonabelian simple group $T$, the invariant $m(T)+1$ is at least $f(|T|)$. As they indicated in their paper, it would be interesting to know if this theorem could be used to prove a similar result for larger families of finite groups. 

For a finite group $G$ and a prime $p$, let $k_{p}(G)$ be the number of conjugacy classes of nontrivial $p$-elements in $G$ and let $n(G)$ be the maximum of $k_{p}(G)$ as $p$ ranges over all the prime factors $p$ of $|G|$. 

Our first main result is an improvement of the above-mentioned theorem of Héthelyi and Külshammer \cite{hethelyi-kulshammer2005}. 

\begin{theorem}
\label{main2}
There exists a function $f(x)$ with $f(x) \to \infty$ as $x \to \infty$ such that $n(G) \geq f(|G|)$ for any finite group $G$.
\end{theorem}

We remark that another precursor to our Theorem \ref{main2} is due to Moret\'o and Nguyen \cite[Theorem 1.1]{MN}. They bound ``large parts of $|G|$" from above in terms of $k_{p'}(G)$ only. More precisely, they show that if $O_{\infty}(G)$ denotes the solvable radical of $G$, then $|G/O_{\infty}(G)|$ is $k_{p'}(G)$-bounded, and $|O_{\infty}(G)/F(G)|$ is metabelian by  $k_{p'}(G)$-bounded. 
Note that in Passman \cite[Corollary 3.5]{passman} another result along these lines is proved which is attributed to Guralnick and states that
$|G/O^{p'}(O_{\infty}(G))|$ is $k_{p'}(G)$-bounded. Theorem \ref{main2} complements these lower bounds for ``large parts of $|G|$" and shows that replacing the slightly more ``local" $k_{p'}(G)$ by the more global $n(G)$ allows for a bound of all of $|G|$, not just portions of $|G|$.

There are many improvements of Landau's theorem in the literature. Here we give one example. The solvable conjugacy class graph $\Gamma(G)$ of a finite group $G$ is defined to be the graph with vertex set $\{ x^{G} : 1 \not= x \in G \}$ with an edge between vertices $x^G$ and $y^G$ if there are elements $x' \in x^{G}$ and $y' \in y^{G}$ with $\langle x', y' \rangle$ solvable. Bhowal, Cameron, Nath, Sambale \cite[Theorem 3.5]{BCNS} recently proved the following generalization of Landau's theorem: for any positive integer $d$, there are only finitely many finite groups $G$ such that the clique number of $\Gamma(G)$ is $d$. 

We now turn to the second topic of this paper. Let $p$ be a prime. In \cite{hethelyi-kulshammer2000} H\'{e}thelyi and K\"{u}lshammer proved that for solvable groups $G$ of order divisible by $p$
we have $k(G)\geq 2\sqrt{p-1} $. This was the origin of a by now extensive and still growing body of literature on generalizations and strengthenings of all kinds. For example, the second author \cite{Keller2} extended the bound to arbitrary groups for large $p$, 
and the third author \cite{Attila2016} proved it for arbitrary groups and all primes.
One of the more recent results is due to Hung and the third author of this paper \cite[Theorem 1.1]{HM} who proved that if $p$ divides the order of a finite group $G$, then
\[  k_{p}(G) + k_{p'}(G)  \geq   2\sqrt{p-1}    \]
with equality if and only if $\sqrt{p-1}$ is an integer and $G=C_p\rtimes C_{\sqrt{p-1}}$ is a Frobenius group (when $p>2$) or $G = C_2$ (when $p=2$).  Here $k_{p'}(G)$ denotes the number of conjugacy classes of elements of $G$ whose orders are coprime to $p$.
In  \cite{hethelyi-kulshammer2000} H\'{e}thelyi and K\"{u}lshammer also conjectured that for any finite group $G$ the number $k(B)$ of complex irreducible characters in a $p$-block $B$ of $G$ is $1$ or is at least $2 \sqrt{p-1}$. This was proved for the principal block $B = B_{0}$ by Hung and Schaeffer Fry \cite{HSch}. 
In a somewhat similar spirit, Hung, Sambale and Tiep \cite[Theorem 1.1]{HungSambaleTiep} proved that if $p$ is a prime dividing the order of a finite group $G$ and all nontrivial $p$-elements are conjugate in $G$, then one of the following holds.
(i) $k_{p'}(G)\geq p$; (ii) $k_{p'}(G) = p-1$ and $G=C_p\rtimes C_{p-1}$ is a Frobenius group when $p$ is odd and is $C_2$ when $p=2$; (iii) $p=11$, $k_{p'}(G)=9$,
and $G$ is the minimal nonsolvable Frobenius group, that is, $G=C_{11}^2\rtimes \mbox{SL}(2,5)$.

The idea of studying $k_{p}(G)$ and $k_{p'}(G)$ for a finite group $G$ is at the heart of our second main result. 
While the bound $2\sqrt{p-1}$ in some of the earlier results above is sharp for some special primes and certain kinds of Frobenius groups, there has been a latent feeling in the community that something is still missing with regards to these lower bounds, or, as Hung, Sambale and Tiep state in
\cite{HungSambaleTiep} that:  ``As it
is obvious from the bound itself that equality could occur only when $p-1$ is a
perfect square, a `correct' bound remains to be found".
The correct bound was recently found by the first and second authors of this paper in \cite{ck} as they noticed that one must
take the arithmetic structure of $p-1$ into account.
Namely, if we write $p-1=ab$ for positive integers $a$ and $b$ with minimal distance (that is, $|a-b|$ is minimal), then they conjecture for a finite group $G$ that   $k(G)\geq a+b$ with equality if and only if $G = C_p \rtimes C_{a}$ or $G = C_p \rtimes C_{b}$, with $C_G(C_p) = C_p$. In \cite{ck} this conjecture is proved for large primes $p$ (using the McKay Conjecture for non-$p$-solvable groups) and for solvable groups $G$. For solvable groups, however, it had already been observed much earlier by H\'{e}thelyi and K\"{u}lshammer in
\cite[Remark (ii)]{hethelyi-kulshammer2003}.

All these observations motivated us to prove a strengthened form of the previous conjecture.

\begin{theorem}\label{main}
Let $G$ be a finite group and $p$ a prime dividing $|G|$. One of the following holds.
\begin{enumerate}
\item[(i)] There exists a factorization $p-1=ab$ with $a$ and $b$ positive integers such that $k_{p}(G)\geq a$ and $k_{p'}(G)\geq b$, with equality in both cases if and only if $G=C_p \rtimes C_b$ such that $C_G(C_p) = C_p$.
\item[(ii)] $p=11$ and $G=C_{11}^2\rtimes \text{\rm SL}(2,5)$.
\end{enumerate}
\end{theorem}

Theorem \ref{main} was already proved in some important special cases a long time ago. The inequalities $k_{p}(G)\geq a$ and $k_{p'}(G)\geq b$ follow immediately from Brauer's work as stated in \cite[Theorem 11.1]{Nav} in the case that $G$ has a Sylow $p$-subgroup of order $p$, by noting the well-known fact that $k_{p'}(G)$ is the number of irreducible $p$-Brauer characters of $G$. Theorem \ref{main} was also known for all groups $G$ with $k_p(G) \leq 3$ (see \cite[Theorem 1.1 and Section 2]{HungSambaleTiep}).

The paper is organized as follows. In Section 2, we prove Theorem \ref{main2}. The remaining sections are devoted to the proof of Theorem \ref{main}. In Section 3, we deal with groups having cyclic Sylow $p$-subgroups and for this we use Brauer's work on modular characters. In Section 4 we present three preliminary lemmas. Section 5 proves Theorem \ref{main} for non-$p$-solvable groups. In Section 6 we prove four steps in the case of $p$-solvable groups. Section 7 deals with all primes $p$ less than $47$. In Section 8 we prove a technical result on coprime actions of almost quasisimple linear groups. In Section 9 we prove Theorem \ref{main} in the special case when the irreducible module for the relevant linear group is imprimitive and when the module is induced from a module for a subgroup whose factor group modulo the kernel of the action is metacyclic. In Section 10 we describe certain linear groups having few orbits. Finally, in Section 11 we finish the proof of Theorem \ref{main}.

\section{Proof of Theorem \ref{main2}}

In this section we prove Theorem \ref{main2}. 



We start with nonabelian finite simple groups. For a nonabelian finite simple group $T$ and a prime $p$, let $m_{p}(T)$ be the number of orbits of the automorphism group $\mathrm{Aut}(T)$ of $T$ on the set of nontrivial $p$-elements of $T$. Let $m(T)$ be the maximum of $m_{p}(T)$ taken over all prime divisors $p$ of $|T|$. Clearly, $m(T) \leq n(T)$. 

The following is \cite[Theorem 1.1]{GMP1} (with $m$ here is $m-1$ in their paper).

\begin{theorem}[Giudici, Morgan, Praeger]
\label{newstar}
There exists an increasing function $f_1$ on the set of natural numbers such that whenever $T$ is a nonabelian finite simple group, then $|T| \leq f_{1}(m(T))$.  
\end{theorem}

Let $O_{\infty}(G)$ denote the largest normal solvable subgroup of the finite group $G$. Theorem \ref{newstar} has the following extension. 

\begin{lemma}
\label{lemmanonsol}	
If $G$ is a finite group with $O_{\infty}(G) = 1$ and $f_1$ is as in Theorem \ref{newstar}, then $$|G| \leq n(G)! (f_{1}(n(G)))^{2 n(G)}.$$
\end{lemma}

\begin{proof}
Let $G$ be a finite group with $O_{\infty}(G) = 1$. Let the socle of $G$ be $\mathrm{Soc}(G)$. This is the direct product $S_{1} \times \cdots \times S_{t}$ of nonabelian simple groups $S_{i}$ with $1 \leq i \leq t$ for some integer $t$. Let the kernel of the conjugation action of $G$ on the set $\{ S_{1}, \ldots , S_{t} \}$ be $B$. This is a subgroup of $\mathrm{Aut}(S_{1}) \times \cdots \times \mathrm{Aut}(S_{t})$. Since every finite simple group can be generated by two elements by \cite[Theorem B]{AG}, we have $|B| \leq {|\mathrm{Soc}(G)|}^{2}$. Thus $|B| \leq \prod_{i=1}^{t} {(f_{1}(m(S_{i})))}^{2}$ by Theorem \ref{newstar}. For every index $i$ with $1 \leq i \leq t$, we have $n(G) \geq m(S_i)$. Clearly, $|G/B| \leq t!$. These give 
\begin{equation}
\label{ee22}	
|G| \leq t! {(f_{1}(n(G)))}^{2t}. 
\end{equation}

Next we show that $t\leq n(G)$. Every group $S_i$ has even order by the Feit-Thompson theorem. For each $j$ with $1 \leq j \leq t$, let $g_j$ be an element of $\mathrm{Soc}(G)$ of order $2$ such that $g_j$ projects onto exactly $j$ factors nontrivially. These elements are pairwise non-conjugate. The claim follows.   

Inequalities (\ref{ee22}) and $t \leq n(G)$ give the statement of the lemma. 
\end{proof}

We continue with an elementary lemma. 

\begin{lemma} Let $N$ be a normal subgroup of a finite group $G$. Then
	\label{newl1}
	\begin{itemize}
		\item[(i)] $n(N) \leq n(G)\cdot|G:N|$ and
		\item[(ii)] $n(G/N) \leq n(G)$.
	\end{itemize} 
\end{lemma}

\begin{proof}
	Let $p$ be a prime for which $k_{p}(N) = n(N)$. There are at least $k_{p}(N)/|G:N|$ conjugacy classes of nontrivial $p$-elements in $G$ lying inside $N$. Since $k_{p}(G) \leq n(G)$, part (i) of the lemma follows. The first two paragraphs of the proof of \cite[Lemma 1.2]{hethelyi-kulshammer2005} give $1 + k_{p}(G/N) \leq 1 + k_{p}(G)$ for every prime $p$. Part (ii) follows. 
\end{proof}	

We recursively define real valued functions $F_{n}(x)$ for real numbers $x$ for every nonnegative integer $n$. Let $F_{0}(x) = x$. For every positive integer $n$, let $\log_{2} (F_{n}(x))$ be $F_{n-1}(x)$. We will need the following technical lemma. 

\begin{lemma}
\label{newl2}	
Let $N$ be a normal subgroup in a finite group $G$. Let $f_1$ and $f_2$ be monotone increasing functions on the set of natural numbers such that $|N| \leq f_{1}(n(N))$ and $|G/N| \leq f_{2}(n(G/N))$. Then
$$
|G| \leq f_{1}(  f_{2}(n(G)) n(G)  )  f_{2}(n(G)).
$$ In particular, if $|N| \leq F_{n}(n(N))$ and $|G/N| \leq F_{m}(n(G/N))$ for positive integers $n$ and $m$, then $|G| \leq F_{n+m+2}(n(G))$.
\end{lemma}

\begin{proof}
We have $n(G/N) \leq n(G)$ and $$n(N) \leq |G/N| \cdot n(G) \leq f_{2}(n(G/N)) \cdot n(G) \leq f_{2}(n(G)) \cdot n(G)$$ by Lemma \ref{newl1}. Thus, $|N| \leq f_{1}(  f_{2}(n(G)) \cdot n(G)  )$ and $|G/N| \leq f_{2}(n(G))$. The first claim follows. Let us prove the second claim. It is sufficient to see that $F_{n}(F_{m}(x) \cdot x) F_{m}(x) \leq F_{n+m+2}(x)$ for all natural numbers $x$. Clearly, $$F_{m}(x) \cdot x \leq 2^{F_{m}(x)} = F_{m+1}(x)$$ and $F_{n}(F_{m+1}(x)) = F_{n+m+1}(x)$. Finally, $F_{n+m+1}(x) \cdot F_{m}(x) \leq F_{n+m+2}(x)$.	
\end{proof}	

Lemma \ref{newl2} is needed in the proof of the following lemma. 

\begin{lemma}
\label{l3}	
Let $G$ be a finite group and let $\ell$ be a positive integer. If there is a chain $1 = N_{0} < N_{1} < \cdots < N_{\ell} = G$ of normal subgroups $N_{0}$, $N_{1}, \ldots , N_{\ell}$ in $G$ such that $N_{i+1}/N_i$ is nilpotent for every index $i$ with $0 \leq i \leq \ell-1$, then $|G| \leq F_{6 \ell -2}(n(G))$.
\end{lemma}

\begin{proof}
Let $G$ be a finite group and let $t$ be a nonnegative integer. We start with a claim. If there is a chain $1 = N_{0} < N_{1} < \cdots < N_{2^{t}} = G$ of normal subgroups $N_{0}$, $N_{1}, \ldots , N_{2^{t}}$ in $G$ such that $N_{i+1}/N_i$ is nilpotent for every index $i$ with $0 \leq i \leq 2^{t}-1$, then $|G| \leq F_{2^{t+2} + 2^{t+1} -2}(n(G))$.	
	
We proceed to prove the claim. We argue by induction on $t$. 

Let $t=0$. In this case $G$ is nilpotent. If $P$ is a Sylow $p$-subgroup of $G$ for some prime $p$, then the number of conjugacy classes of $P$ is at least $\log_{2}|P|$, and so $|P| \leq F_{1}(n(P)+1) \leq F_{1}(n(G)+1)$. If $G$ is the direct product of its Sylow $p_i$-subgroups $P_i$ with $i$  satisfying $1 \leq i \leq k$ for some integer $k$ and the primes $p_i$ satisfying $p_{1} < \ldots < p_{k}$, then $$|G| \leq {(F_{1}(n(G)+1))}^{k} \leq {(F_{1}(n(G)+1))}^{p_{k} - 1} \leq {(F_{1}(n(G)+1))}^{n(G)} \leq F_{4}(n(G)).$$

Let $t > 0$. Assume that the claim is true for $t-1$. Let $N = N_{2^{t-1}}$. We have $|N| \leq F_{2^{t+1} + 2^{t} -2}(n(N))$ and $|G/N| \leq F_{2^{t+1} + 2^{t} -2}(n(G/N))$ by the induction hypothesis. This gives $|G| \leq F_{2^{t+2} + 2^{t+1} -2}(n(G))$ by Lemma \ref{newl2}, which proves the claim above. 	

Let $G$ be a finite group and let $\ell$ be a positive integer. Assume that there is a chain $1 = N_{0} < N_{1} < \cdots < N_{\ell} = G$ of normal subgroups $N_{0}$, $N_{1}, \ldots , N_{\ell}$ in $G$ such that $N_{i+1}/N_i$ is nilpotent for every index $i$ with $0 \leq i \leq \ell-1$. Let the binary expansion of $\ell$ be $\ell = \sum_{i=1}^{m} 2^{t_i}$ where each $t_i$ is a nonnegative integer. We will prove the lemma by induction on $m$. The previous claim gives the result for $m=1$. Let $m > 1$ and assume that the conclusion holds for $m-1$. Write $r = \sum_{i=1}^{m-1} 2^{t_i}$. The induction hypothesis provides $|N_r| \leq F_{6 r -2}(n(N_{r}))$ and $|G/N_{r}| \leq F_{6 (\ell - r) -2}(n(G/N_{r}))$ by the claim. The result now follows from Lemma \ref{newl2}.
\end{proof}

The next lemma deals with solvable groups.  

\begin{lemma}
\label{lemmasolvable}
If $G$ is a finite solvable group, then $$|G| < F_{23}(  (n(G)+1)^{149}  ).$$
\end{lemma}

\begin{proof}
Let $G$ be a finite solvable group. Let $\Phi(G)$ and $F(G)$ be the Frattini and Fitting subgroups of $G$ respectively. The group $F(G/\Phi(G)) = F(G)/\Phi(G)$ is a completely reducible and faithful $G/F(G)$-module (possibly of mixed characteristic) and $G/\Phi(G)$ splits over $F(G)/\Phi(G)$ by  a theorem of Gaschütz (see \cite[Theorem 1.12]{manz}).   

Put $H = G/\Phi(G)$ and $V = F(G)/\Phi(G)$. Let $R$ be a subgroup of $H$ such that $H =RV$ and $R \cap V = 1$. Let the completely reducible $H$-module $V$ have $t$ irreducible summands, $V_{1}, \ldots , V_{t}$. Let $H_{i} = H/C_{H}(V_{i})$ for every $i$ with $1 \leq i \leq t$. For each $i$, let $n(H_{i},V_{i})$ be the number of orbits of $H_{i}$ on $V_{i}$. 

Fix an index $i$ with $1 \leq i \leq t$. There are two cases to consider by a theorem of the second author \cite[Theorem 2.1]{Keller1}. 

In the first case, $|V_{i}| \leq n(H_{i},V_{i})^{37}$. Since $|H_{i}V_{i}| \leq |V_{i}|^{4}$, by the theorems of Pálfy \cite[Theorem 1]{Palfy} and Wolf \cite[Theorem 3.1]{Wolf}, we get $|H_{i}V_{i}| \leq  n(H_{i},V_{i})^{148}$. 

In the second case, there are normal subgroups $A_{i}$ and $B_{i}$ of $H_{i}V_{i}$ such that $V_{i} \leq A_{i} \leq B_{i} \leq H_{i}V_{i}$, the factor groups $A_{i}/V_{i}$ and $B_{i}/A_{i}$ are abelian and $H_{i}V_{i}/B_i$ may be considered as a permutation group of degree $k_{i}$ at most $(1/5) \log_{3} n(H_{i},V_{i})$. Since $H_{i}V_{i}/B_{i}$ is solvable, $|H_{i}V_{i}/B_{i}| \leq 24^{(k_{i}-1)/3}$ by \cite{Dixon}, and so $$|H_{i}V_{i}/B_{i}| < 24^{\log_{3}(n(H_{i},V_{i}))/15} < n(H_{i},V_{i}).$$ 

Observe that $n(H) \geq \max_{1 \leq i \leq t} \{ n(H_{i},V_{i}) \} - 1$ and $n(G) \geq n(H)$ by Lemma \ref{newl1}. 

Observe from the above that there is a normal subgroup $N$ in $G$ such that $N$ has a chain of nilpotent normal subgroups of length at most $4$, so that $|N| \leq F_{22}(n(N))$ by Lemma \ref{l3}, and $|G/N| \leq n(H_{i},V_{i})^{148} \leq (n(G)+1)^{148}$. Since $$n(N) \leq n(G) \cdot |G:N| < (n(G)+1)^{149}$$ by Lemma \ref{newl1}, we conclude that $$|G| = |N||G/N| \leq F_{22}(n(N))(n(G)+1)^{148} < F_{22}(  (n(G)+1)^{149}  )(n(G)+1)^{148},$$ which is less than $F_{23}(  (n(G)+1)^{149}  )$.
\end{proof}

We are now in the position to prove Theorem \ref{main2}. 

\begin{proof}[Proof of Theorem \ref{main2}]
Let $G$ be a finite group. Put $N = O_{\infty}(G)$. We have $|N| < F_{23}(  (n(N)+1)^{149}  )$ by Lemma \ref{lemmasolvable} and $|G/N| \leq   n(G/N)! (f_{1}(n(G/N)))^{2 n(G/N)}$ by Lemma \ref{lemmanonsol}. From these we may obtain an increasing function $g$ on the set of possible values of $n(G)$, by Lemma \ref{newl2}, such that $|G| \leq g(n(G))$. Let $g^{-1}$ denote the inverse function of $g$. This is also an increasing function. We have $n(G) \geq g^{-1}(|G|)$ (whenever $|G|$ is in the domain of $g^{-1}$). Let $f$ be the function defined on the set of natural numbers such that $f(x)$ is equal to $g^{-1}(y)$ where $y$ is the smallest member of the domain of $g^{-1}$ that is at least $x$. We have $n(G) \geq f(|G|)$. This completes the proof of the theorem.  
\end{proof}

\section{Groups with cyclic Sylow $p$-subgroups}

We begin the proof of Theorem \ref{main}.

Let $p$ be a prime and let $P$ be a Sylow $p$-subgroup of a finite group $G$. Suppose that $p \mid |G|$, that is $P \not= 1$. In this section we prove Theorem \ref{main} in case $P$ is cyclic and in case $Z(P)$ has an element of order $p^{2}$.


\begin{lemma}\label{c}
    Let $G$ be a finite group and $p$ a prime dividing $|G|$. If $q$ is a prime with $q\geq p$ and $Q$ is a Sylow $q$-subgroup of $G$ such that $Z(Q)$ contains an element of order $q^2$, then $k_q(G)\geq q+1$. In particular, depending on whether $q=p$ or $q>p$,
    we obtain $k_p(G)>p$ or $k_{p'}(G)>p$ (respectively).
\end{lemma}

\begin{proof}
By the hypothesis, there exists a subgroup $U$ and a Sylow $q$-subgroup $Q$ of $G$
such that $U\leq Z(Q)$, $U$ is cyclic, and $|U|=q^2$. Now let $\mathcal{T}$ be the set of those nontrivial conjugacy classes of $G$ which have a non-empty intersection with $U$, that is, $\mathcal{T}$ consists of those conjugacy classes $x^G$ for which $1\neq x\in G$ and $x^G\cap U\neq \emptyset$.
Now let $K\in \mathcal{T}$. Then we can write $K=x^G$ for some $x$ in $U$, and the order of
$x$ is $q^k$ for some $k\in\{1,2\}$.

Now if $g \in G$ such that $x^g \in U$, then $g$ normalizes the subgroup $U_0=\langle x\rangle$ of $U$ since $U$ has a unique subgroup of order $q^k$. But since $Q\leq C_G(U)$, we know that $N_G(U_0)/C_G(U_0)$ is a $q'$-group whose order divides $|\mbox{Aut}(U_0)|\in\{q(q-1), q-1\}$, that is, $|N_G(U_0)/C_G(U_0)|$ divides $q-1$. This shows that at most $q-1$ elements of $K$ can be in $U$.

By this argument we see that each conjugacy class in $\mathcal{T}$ has at most $q-1$ elements in $U$, which implies that
\[k_q(G) \geq |T|\geq \frac{|U|-1}{q-1}=\frac{q^2-1}{q-1}=q+1,\]
as desired. The remainder of the statement of the lemma now follows immediately.
\end{proof}

\begin{theorem}
\label{Sylowp}
Let $G$ be a finite group and let $p$ be a prime dividing $|G|$. Let $P$ be a Sylow $p$-subgroup of $G$. Assume that $P$ is cyclic or $Z(P)$ has an element of order $p^{2}$. There exists a factorization $p-1 = ab$ with $a$ and $b$ positive integers such that $k_{p}(G) \geq a$ and $k_{p'}(G) \geq b$ with equalities in both cases if and only if $G=C_p \rtimes C_b$ such that $C_G(C_p) = C_p$.
\end{theorem}

\begin{proof}
Let $G$ be a finite group with a Sylow $p$-subgroup $P$ which is cyclic or that $Z(P)$ contains an element of order $p^{2}$. If $|P|\geq p^2$, then we can
apply Lemma \ref{c}, with $q=p$ and thus, obtain that $k_p(G)>p$. Hence we can take $a=p-1$ and $b=1$, which proves the theorem with $k_p(G)>a$ and $k_{p'}(G)\geq b$. In particular,
there is no case of equality here.

Let us assume that $|P|=p$. If $p=2$, then $k_{2}(G) \geq 1$ and $k_{2'}(G) \geq 1$ with equality in both cases if and only if $G = P$. Let $p$ be odd. Let $C$ be the centralizer and $N$ the normalizer of $P$ in $G$. Let $b = |N/C|$. We have $k_{p}(G) \geq (p-1)/b$ by Sylow's theorem. The number $k_{p'}(G)$ is equal to the number of irreducible Brauer characters in $G$. This number is at least the number of irreducible Brauer characters in the principal block $B_{0}$, which in turn is equal to $b$ by \cite[Theorem 11.1 (c)]{Nav}. This proves $k_{p}(G) \geq a$ and $k_{p'}(G) \geq b$ where $a = (p-1)/b$. It remains to describe all possibilities when there are equalities in both cases. Let $k_{p}(G) = a$ and $k_{p'}(G) = b$. We certainly have $k_{p}(G) \geq m$ and $k_{p'}(G) \geq n$ for some factorization $p-1 = mn$ by considering the principal block of $G$ as before. This forces $a = m$ and $b = n$. If $G$ has more than one $p$-block, then there are at least $b+1$ irreducible Brauer characters in $G$, which is a contradiction. Let $G$ have a unique $p$-block. For $p$ odd, this happens, by \cite[Theorem 1 (a)]{Harris}, if and only if the generalized Fitting subgroup of $G$ is $O_{p}(G)$. In our situation $O_{p}(G)$ (which is self-centralizing) is the Sylow $p$-subgroup (of order $p$) of $G$. This implies that $G \cong C_{p} \rtimes C_{b}$. This is the group mentioned in the statement of the theorem.
\end{proof}

\section{Three lemmas}

In this section we collect three lemmas. 

\begin{lemma}
\label{nonsol}	
In order to prove Theorem \ref{main} for a prime $p$ and a nonsolvable finite group $G$, we may assume that $k_{p}(G) \geq 2$, $k_{p'}(G) \geq 3$ and $p \geq 7$.	
\end{lemma}	

\begin{proof}
We may assume that $k_{p}(G) \geq 2$ by \cite[Theorem 1.1]{HungSambaleTiep} and that $k_{p'}(G) \geq 3$ by Burnside's theorem. It follows that we may take $p$ to be at least $7$. 		
\end{proof}

The following lemma is \cite[Lemma 7.1]{HM}.  

\begin{lemma}
\label{l1}	
Let $p$ be a prime. Let $N$ be a normal subgroup of a finite group $G$. We have $k_{p'}(G/N) \leq k_{p'}(G)$ and $k_{p}(G/N) \leq k_{p}(G)$.
\end{lemma}

The following lemma will also be useful throughout the paper.

\begin{lemma}
\label{??}	
Let $p$ be a prime. Let $H$ be a finite group of order not divisible by $p$, and let $V$ be a finite $H$-module over the field with $p$
elements. (We do not require $V$ to be faithful or irreducible.) Write $HV$ for the semidirect product of $H$ and $V$ with respect 
to the action of $H$ on $V$. Then $k_{p'}(HV) =k(H)$. 
\end{lemma}

\begin{proof}
It is easy to see that  $k_{p'}(HV) \geq k(H)$. To show that  $k_{p'}(HV) \leq k(H)$, it suffices to show that every $p'$-element 
of $HV$ is conjugate in $HV$ to some element in $H$. To do so, it suffices to show that if $g\in H$ and
$v\in V$ such that $gv$ is a $p'$-element in $HV$, then there exists a $w\in V$ such that $(gv)^w=g$, that is, $v=w^gw^{-1}\ \ (1)$
(where we view $V$ as a normal subgroup of $HV$ and write its operations multiplicatively). Now note that for $[g, V]$, which is
the subgroup generated by the elements $[g, x]$ for all $x\in V$, we actually have $[g, V]=\{ [g,x]\ |\ x\in V\}\ \ (2)$, and we also have
$w^gw^{-1}=w^{-1}w^g=[w,g]=[g,w]^{-1}\in [g, V]\  \  (3)$.
Now define the map $\phi: V\to V$ by $\phi(x)=xx^gx^{g^2}\dots x^{g^{m-1}}$, where $m$ is the order of $g$.
Observe that $\phi$ is a homomorphism such that $\phi(V)\leq C_V(g)$. Write $W$ for the kernel of $\phi$. Now consider
the map $\alpha: C_V(g)\to C_V(g)$ defined simply as the restriction of $\phi$ to $C_V(g)$. Then $\alpha(x)=x^m$, and since
$p$ does not divide $m$, we see that the kernel of $\alpha$ is trivial. Hence $\alpha$ is injective and thus also surjective. This
shows that $\phi(V) = C_V(g)$ and hence $\dim W =\dim V -\dim  C_V(g)$.  Since by coprime action we also have the
well-known decomposition $V=[g,V] \times C_V(g)$, we obtain that $\dim W=\dim [g,V]$. Moreover, since for $x\in V$ we have
$\phi([g,x])=\phi(v^{-g}v)=\phi(v^{-g})\phi(v)=1$, it follows that $[g,V]\leq W$. Hence altogether $[g,V]=W\ \ (4)$. \\
Recall that we want to find a $w\in V$ satisfying $(1)$, But by $(2), (3), (4)$ all we have to do is to show that $v\in W$.
Now $(gv)^m=g^mvv^g\dots v^{g^{m-1}}=\phi(v)\in V$ (since $g^m=1$). As $p$ does not divide the order of $gv$, this
forces $(gv)^m=\phi(v)=1$. Hence $v\in W$ and the proof is complete.
\end{proof}

\section{Non-$p$-solvable groups}

In this section, we prove Theorem \ref{main} in case $G$ is not a $p$-solvable group. 

We first deal with almost simple groups.

\begin{lemma}
\label{almostsimple}
Let $G$ be an almost simple group with socle $S$. Let $p$ be a prime divisor of the order of $S$. There exists a factorization $p-1 = ab$ with $a$ and $b$ positive integers such that $k_{p}(G) \geq a$ and $k_{p'}(G) \geq b$. Equalities in both inequalities cannot occur at the same time.
\end{lemma}

\begin{proof}
Let $P$ be a Sylow $p$-subgroup of $G$. If $P$ is cyclic, then the result follows from Theorem \ref{Sylowp}. Assume that $P$ is not cyclic. If $p$ does not divide $|G/S|$, then $k_{p'}(G) \geq p$ by \cite[Theorem 6.2]{HungSambaleTiep}. Assume also that $p$ divides $|G/S|$. Let $k_{p}(G) \geq 2$, $k_{p'}(G) \geq 3$ and $p \geq 7$. This assumption can be made by Lemma \ref{nonsol}.

Since $k_{p'}(G) \geq k_{p'}(G/S)$ and $k_{p}(G) \geq k_{p}(G/S) +1$ by Lemma \ref{l1} and the fact that $p$ divides $|S|$, it would be sufficient to show that Theorem \ref{main} is true for the group $G/S$. The factor group $H = G/S$ is a subgroup of $\mathrm{Out}(S)$. Since $p \geq 7$, the group $S$ must be a simple group of Lie type. Let $Q$ be a Sylow $p$-subgroup of $H$. If $Q$ is cyclic, then the result follows from Theorem \ref{Sylowp}. Assume that $Q$ is not cyclic. Since $p \geq 7$, by inspecting the structure of $\mathrm{Out}(S)$ (see \cite[Theorem 2.5.12]{GLS}) it remains to deal with the cases where  $S$ is a projective special linear group or a projective special unitary group.

Let $q=\ell^{f}$ where $\ell$ is the defining characteristic of $S$ and $f$ is a positive integer. Assume first that the rank $r$ of $S$ is at least $2$. For a projective special linear group $S$, we have $\mathrm{Out}(S)=C_{(r+1,q-1)}\rtimes (C_f\times C_2)$ and  for a projective special unitary group $S$, we have $\mathrm{Out}(S)=C_{(r+1,q+1)}\rtimes C_{2f}$ (see \cite[Theorem 2.5.12]{GLS} and the discussion following the proof). Since $p\geq 7$ and $Q$ is not cyclic, $p$ must divide $(r+1,q-1)$ and $(r+1,q+1)$, respectively. Thus, $r\geq p-1\geq 6$ and $q\geq 7$ in both cases. With these restrictions one checks that $k_{p'}(G)> k_{p'}(S)/|\mathrm{Out}(S)|\geq r\geq p-1$ using Theorem 1.4 in \cite{HM}. A similar argument shows that the case  $r=1$ cannot occur since $|\mathrm{Out}(S)|\in\{f, \: 2f\}$ and  $Q$ is not cyclic.
\end{proof}

\begin{theorem}
\label{non}
Let $p$ be a prime and let $G$ be a finite group which is not $p$-solvable. There exists a factorization $p-1 = ab$ with $a$ and $b$ positive integers such that $k_{p}(G) \geq a$ and $k_{p'}(G) \geq b$. Equalities in both inequalities cannot occur at the same time.
\end{theorem}

\begin{proof}
Let $S$ be a nonabelian simple composition factor of $G$ whose order is divisible by $p$. Let $M$ and $N$ be normal subgroups in $G$ such that $M > N$ and $M/N$ is isomorphic to $S_{1} \times \cdots \times S_{t}$ where each $S_i$ is isomorphic to $S$. We may assume that $N = 1$ by Lemma \ref{l1}.

Let $t \geq 2$. Let the number of orbits of $\mathrm{Aut}(S)$ on the set of $p'$-elements in $S$ be $c$. In this case $k_{p'}(G) \geq \binom{t+c-1}{t}$ by the proof of \cite[Lemma 4.3]{MS}. We have $c \geq \sqrt{p-1}$ by \cite[Theorem 2.1 (iii)]{HM} and \cite[Table 1]{HM}. Thus, $$k_{p'}(G) \geq \binom{t+c-1}{t} \geq \binom{c+1}{2} > \frac{c^{2}}{2} = (p-1)/2.$$ Since we may assume that $k_{p}(G) \geq 2$ by Lemma \ref{nonsol}, the result follows.

Let $t=1$. The group $S = S_1$ is normal in $G$. We may assume that the centralizer of $S$ in $G$ is trivial by Lemma \ref{l1}. It follows that $G$ is almost simple with socle $S$. The result follows from Lemma \ref{almostsimple}.
\end{proof}

\section{Four steps}

In this section we continue with the proof of Theorem \ref{main}. 

Let $G$ be a counterexample to Theorem \ref{main} with $|G|$ minimal. We suppose that $G$ is a $p$-solvable group. We have that  $(G,p) \not= (C_{11}^2\rtimes \text{\rm SL}(2,5), 11)$, otherwise Part (ii) of Theorem \ref{main} holds. Let $V$ be a minimal normal subgroup in $G$. In this section we will prove four properties of $G$ given in four steps.  

We know from Lemma \ref{l1} that $k_p(G/V)\leq k_p(G)$ and $k_{p'}(G/V)\leq k_{p'}(G)$. Observe that $k_p(G/V) < k_p(G)$ if $p$ divides $|V|$, and  $k_{p'}(G/V) < k_{p'}(G)$ if $p  \nmid |V|$.

\medskip

{\it Step 1. $V$ is an elementary abelian $p$-group of rank at least $2$ and $p$ does not divide $|G/V|$.}

\medskip

Assume that $p$ divides $|G/V|$. By the fact that $G$ is a minimal counterexample, we know that $G/V$
satisfies either Part (i) or Part (ii) of Theorem \ref{main}. 

Let $p=11$ and $G/V\cong (C_{11})^{2} \rtimes \mbox{SL}(2,5)$. In this case $k_{11}(G/V)=1$ and $k_{11'}(G) \geq k_{11'}(G/V) \geq k(\mbox{SL}(2,5)) \geq 9$. If $11  \mid |V|$, then $k_{11}(G)\geq 2$
contradicting $G$ being a counterexample, and if $11  \nmid |V|$, then it is easy to see that this forces $k_{11'}(G) \geq 11$, again implying that $G$ is not a counterexample.

It remains to consider the case that $G/V$ satisfies Part (i) of Theorem \ref{main}.
Then there exist positive integers $a$ and $b$ such that $p-1=ab$ and $a\leq k_p(G/V) \leq k_p(G)$ and $b\leq k_{p'}(G/V)\leq k_{p'}(G)$. If $p\mid |V|$, then even $k_p(G/V) < k_p(G)$, and we obtain a contradiction. If $p  \nmid |V|$, then $a\leq k_p(G)$ and $b\leq k_{p'}(G/V)< k_{p'}(G)$, which is again a contradiction. 

Thus we have shown that $p$ does not divide $|G/V|$. Therefore the prime $p$ must divide $|V|$ and so $V$ is an elementary abelian $p$-group. The size of $V$ must be at least $p^{2}$ by Theorem \ref{Sylowp}.

\medskip

{\it Step 2. $G =VH$ for a subgroup $H$ of $G$ of order coprime to $p$ and $H$ acts faithfully on $V$. The group $H$ acts irreducibly on $V$.}

\medskip

Since $V$ is a minimal normal subgroup of $G$, the second claim follows. We claim that $V$ is the unique minimal normal subgroup of $G$. Let $M$ be another minimal normal subgroup of $G$. It is well known that $G$ is isomorphic to a subgroup of $G/V\times G/M$. But $|G/V|$ and $|G/M|$ are not divisible by $p$ by the previous paragraph. This contradicts the fact that $|G|$ is divisible by $p$. It follows from Step 1 and the Schur-Zassenhaus theorem that $G$ splits over $V$, that is, $G=VH$ for a subgroup $H$ of $G$ of order coprime to $p$. Moreover, $H$ acts faithfully on $V$.

\medskip

{\it Step 3. If $|V| = p^{2}$, then $H$ is not solvable.}

\medskip

Let $G$ be solvable. Since $V$ is the unique minimal normal subgroup of $G$, we may view $H$ as an irreducible $p'$-subgroup of $\textrm{GL}(2, p)$, and the structure of $H$ is described in \cite[Theorem 2.11]{manz}. \\ 
In cases (a) and (b) in \cite[Theorem 2.11]{manz}, we can conclude that $H$ contains an abelian
normal subgroup $X$ such that $x := |X|\leq  p^2 - 1$ and $|H : X| \leq 2$.

First suppose that $H=X$. Then $H$ acts frobeniusly on $V$. It follows that
$$
k_{p'}(G)=k(H)=x
\quad\text{and}\quad
k_p(G)=n(H, V)-1 =\frac{|V|-1}{|H|}=\frac{p^2-1}{x}.
$$
If $x> p-1$, then we can choose $a=1$ and $b=p-1$, which is a contradiction. If $x\leq p-1$, then $(p^2-1)/x\geq p+1>p-1$ and we choose $a=p-1$ and $b=1$, contradiction.

Let us assume that $|H:X|=2$. Then
$$
k_{p'}(G)= k(H)\geq \frac{|X|}{|H:X|}=\frac{x}{2}
\quad\text{and}\quad
k_p(G)\geq n(H, V)-1\geq \frac{|V|-1}{|H|}=\frac{p^2-1}{2x}.
$$
If $x>p-1$, we may choose $a=2$ and $b=(p-1)/2$, contradiction. If $x\leq p-1$, then
$$
k_p(G)\geq \frac{p^2-1}{2(p-1)}=\frac{p+1}{2}>\frac{p-1}{2}
$$
and we choose $a=(p-1)/2$ and $b=2$, contradiction. (Note that we may assume that $H>1$ and thus, $k_{p'}(G) \geq 2$.)

We are now in Case (c) in \cite[Theorem 2.11]{manz}, and  we follow the proof of Theorem in \cite{hethelyi-kulshammer2003} by adjusting it to our hypothesis. Then $F(H) = Q \circ X$ (central product)
and $Q \cap X = Z(Q)$, where $Q\cong Q_8$, which is normal in $H$, and $X = Z(H)$ is cyclic and $|X|$ divides $p - 1$. Moreover, $H/F(H)$ acts irreducibly on
$Q/Z(Q)$, so $H/F (H) \cong Z_3$ or $H/F (H)\cong S_3$. Let $x := |X : Z(Q)|=|X|/2$, and then $x$ divides $(p-1)/2$ since $|X|$ divides $p - 1$. Let $(p-1)/2=xk$ for some positive integer $k$. If we count the irreducible characters of $H$ as in the proof of Theorem in \cite{hethelyi-kulshammer2003}, then we find that $k(H) = 7x$, $|H| = 24x$ in case $H/F(H)\cong Z_3$, and $k(H) = 8x$, $|H| = 48x$  in case $H/F(H)\cong S_3$. In the first case, we have \begin{equation}\label{30} k_{p'}(G)= k(H)=7x=\frac{7(p-1)}{2k}>\frac{2(p-1)}{2k}=\frac{p-1}{k}\end{equation} and \begin{equation}\label{40} k_p(G)\geq n(H, V)-1\geq \frac{|V|-1}{|H|}\geq\frac{p^2-1}{24x}>\frac{p-1}{2x}=k,\end{equation}  where the last inequality holds  for $p> 11$. Thus, we choose $a=k$ and $b=(p-1)/k$ for the prime $p> 11$, contradiction. One can find suitable $a$ and $b$ for primes $p\leq 11$ by considering Inequalities (\ref{30}), (\ref{40}) and the fact that the integer $x$ divides $(p-1)/2$, which gives us a contradiction.

Similar calculations lead to a contradiction that we have Theorem \ref{main} for the prime $p>23$ in case  that $k(H) = 8x$ by replacing $7x$ by $8x$ in Inequality (\ref{30}) and $24x$ by $48x$ in Inequality (\ref{40}). For the prime $p\leq 23$, we again find the suitable $a$ and $b$ in Theorem \ref{main}, contradiction.

\medskip

{\it Step 4.  We will show that $|V| \geq p^{3}$.}

\medskip

We may assume that $|V| = p^{2}$ and $H$ is not solvable by Steps 1 and 3. In this case,  by \cite[Section XII.260]{dickson} or
\cite[II, Hauptsatz 8.27]{huppert} we know that $p\equiv \pm 1 (\text{mod } 10)$ and also that either $H/Z(H)\cong A_5$ or $H/Z(H)\cong S_5$ (given that $(|H|, |V|)=1$). Thus we write $|H|=60x$, where $x=1$ or 2 depending on whether $H/Z(H)\cong A_5$ or $H/Z(H)\cong S_5$, respectively.\\\\
First suppose that $p>60$.

Now write $Z=Z(H)$ and consider $V_Z$, that is, $V$ viewed as a $Z$-module. If $V_Z$ is irreducible (i.e., $Z$ acts irreducibly on $V$), then
by \cite[II, Hilfssatz 3.11]{huppert} or \cite[Theorem 2.1]{manz}, $G$ is solvable, a contradiction. Hence $V_Z$
is the direct sum of two $Z$-modules of order $p$, and since clearly $Z$ acts frobeniusly on $V$ (i.e., $ZV$ is a Frobenius group), this forces that $|Z|$ divides $p-1$.
Now clearly $k_{p'}(G) = k(H)\geq |Z|+1$. \\
Then, 
\begin{equation}\label{10}
	k_p(G)\geq n(H, V)-1 \geq \frac{|V|-1}{|H|}\geq\frac{p^2-1}{60x|Z|}=\frac{(p+1)(p-1)}{60x|Z|} > \frac{p-1}{x|Z|}  
\end{equation}
where the last inequality follows as $p>60$. If $x=1$, then we choose  $a= (p-1)/|Z|$ and $b=|Z|>1$, which is a contradiction.
So we may assume that $x=2$, that is, $H/Z(H)\cong S_5$. Then by Inequality (\ref{10}), we see that $k_p(G)> \frac{p-1}{|Z|}$  is still true when $p>120$, and hence we may choose  $a= (p-1)/|Z|$ and $b=|Z|>1$, which is a contradiction. Now we will have a contradiction to  Theorem \ref{main} for the primes $p>60$ in the case where $H/Z(H)\cong S_5$.

Since the Schur multiplier of $S_5$ is $C_2$, we have the following cases: In the first case, $H=(C_2.S_5)\times Z_1$  for a subgroup $Z_1$ of $Z(H)$ with $|Z_1|=\frac{|Z(H)|}{2}$, where $C_2.S_5 $ is the second central stem extension by $C_2$ of $S_5$. Note that $|Z_1|$ is an odd number because $Z(H)$ is cyclic. It follows that we get $k(H)=12.\frac{|Z|}{2}=6|Z|.$ In the latter case, $H=S_5\times Z(H)$, and hence  we have $k(H)=7 |Z|.$

Therefore, we have that

\begin{equation}\label{2}  k_p(G)\geq n(H, V)-1 \geq \frac{|V|-1}{|H|}\geq\frac{p^2-1}{120|Z|}
	\quad \text{and} \quad k_{p'}(G)= k(H)\geq 6|Z|.
\end{equation}

Since $p\equiv \pm 1 (\text{mod } 10)$, the primes between $60$ and $120$ that we have to consider are  $61, 71, 79, 89, 101, 109$.

Let $p=61$. We know that $|Z|$ divides $p-1=60$, and also by \cite[Section 2]{HungSambaleTiep}, we can assume that $\displaystyle|Z|\leq \frac{p-1}{4}$. Thus,
$|Z|\leq 15$, and hence we get $|Z|\in\{2, 3, 4, 5, 6, 10, 12, 15\}$. Let $|Z|=2$. Then by (\ref{2}),  we have $$k_p(G)\geq \frac{p^2-1}{120|Z|}=\frac{31}{|Z|}=15.5 \quad   \textrm{and} \quad   k_{p'}(G)= k(H)\geq 6|Z|=12.$$ Thus, we can choose $(a, b)=(5, 12)$, contradiction.  Again by using the inequalities in (\ref{2}),  we can choose $(a, b)=(5, 12)$ if $|Z|\in\{3, 4, 5, 6\}$,   and   $(a, b)= (1, 60)$ if $|Z|\in \{10, 12, 15\}$, which are  contradictions.

Let $p\in\{71, 79, 89, 101, 109\}$. Then by similar arguments as in the previous paragraph we have the integers $a$ and $b$ in Theorem \ref{main} as follows:

\begin{table}[h!]
\begin{tabular}{|c|c|}
	\hline $p$ & $(a, b)$\\
	\hline   $71$ & $\{ (14, 5), (5, 14), (2, 35), (1, 70)\}$  \\
	$79$  &  $\{ (13, 6), (6, 13), (1, 78)\}$\\
	$89$ & $\{ (22, 4), (4, 22),  (1, 88)\}$\\
	$101$& $\{ (10, 10), (4, 25), (1, 100)\}$\\
	$109$ & $\{ (18, 6), (6, 18), (1, 108)\}$ \\\hline
    \end{tabular}
    \caption{Possible $(a, b)$ pairs}
  \end{table}
  \vspace{-5mm}
\noindent This is a contradiction, which proves Theorem \ref{main} for the primes $p>60$ in the case that $H/Z(H)\cong S_5$. 
Now let us consider the primes $p\leq 60$. Since $p\equiv \pm 1 (\text{mod } 10)$, we get $p\in\{11, 19, 29, 31, 41, 59\}$. We use GAP \cite{GAP} to confirm the result that we can always find the integers $a$ and $b$ as in Theorem \ref{main}, which is a contradiction.

\section{Small primes}
\label{sec:small_primes}

In this section, we show that Theorem \ref{main} holds for every prime $p$ at most $43$.

Assume that Theorem \ref{main} is not true for a prime $p$ at most $43$ and a finite group $G$. Notice that we may assume that $k_p(G) \geq 3$ (by \cite[Theorem 1.1 and Section 2]{HungSambaleTiep}). Moreover, we have $k_{p'}(G)\geq 2$. Thus, we may assume that $p\geq 7$. 

We also know that $G = HV$ for a subgroup $H$ of $G$ of order coprime to $p$ and for an elementary abelian $p$-subgroup $V$ of $G$ which is normal in $G$. Moreover, $V$ is a faithful and irreducible $H$-module of size at least $p^{3}$. These follow from Section 5. 
 
Let $p=7$. If $k_p(G) \geq 4$, then we have Theorem \ref{main} by choosing $a=3$ and $b=2$ because $k_{p'}(G)=k(H)\geq 2$, a contradiction. Thus, $k_p(G)= 3$ and $k_{p'}(G)=2$. By \cite[Example 12.4]{hup}, 
$|H|=2$, which gives us the contradiction that \[k_p(G)= n(H, V)-1\geq\frac{|V|-1}{|H|}\geq 171>3.\]
This kind of contradiction using that  $k_p(G)\geq (|V|-1)/|H|$ will occur many more times in this section.

Let $p=11$. If $k_p(G) \geq 6$, then we have Theorem \ref{main} by choosing $a=5$ and $b=2$ because $k_{p'}(G)=k(H)\geq 2$, a contradiction. Thus, we may assume that $3\leq k_p(G)\leq 5$. If $k_{p'}(G)=k(H)\geq 5$, then we are done. Thus, we assume that $2\leq k_{p'}(G)=k(H)\leq 4$. By \cite[Example 12.4]{hup}, we know that $|H|\leq 12$. Thus we have the contradiction that  $(|V|-1)/|H|\geq \frac{11^3-1}{12}>5.$

Let $p=13$. If $k_p(G)>4$ and $k_{p'}(G)\geq3$, then we are done. Let us assume that $3\leq k_p(G)\leq 4$,   and hence $2\leq k_{p'}(G)=k(H)\leq 4$. By \cite[Example 12.4]{hup}, we have $|H|\leq 12$, and hence we get the contradiction that $(|V|-1)/|H|>4.$

If $p=17$, then we have $3\leq k_p(G)\leq 8$ and so, $2\leq k_{p'}(G)\leq 5$. Thus, we have the contradiction that   $(|V|-1)/|H|>8.$

Let $p=19$. We have $3\leq k_p(G)\leq 9$ and so, $2\leq k_{p'}(G)\leq 6$. Thus, $|H|\leq 72$ by \cite[Remark 12.4]{hup}, which gives us the contradiction that  $(|V|-1)/|H|>9.$

Let $p=23$. We may assume that $3\leq k_p(G)\leq 11$ and so, $2\leq k_{p'}(G)\leq 10$. Thus,  $|H|\leq 20160$ by \cite[Table 1, 2]{Vera1} and also, $|V|=23^3$ because of $3\leq k_p(G)\leq 11$. Since $3\leq k_p(G)\leq 11$, the calculations show by \cite[Table 1, 2]{Vera1}  that $H\cong C_7^2\rtimes \textrm{SL}(2, 3), \:  C_7^2\rtimes \textrm{SL}(2, 3).C_4,  \: A_7,  \: C_{11}^2\rtimes \textrm{SL}(2, 5)$, $M_{11}$ or $\textrm{PSL}(3, 4)$. Since $H$ acts faithfully on $V$, we have that $H$ is a subgroup of $\textrm{GL}(3, 23)$, whose order is not divisible by $5$ and $7^2$. On the other hand, either $5$ or $7^2$ divides the order of $H$, which is a contradiction. 

Let $p=29$.  Hence we may assume that $3\leq k_p(G)\leq 14$ and so, $2\leq k_{p'}(G)\leq 9$. Thus, $|H|\leq 2520$ by \cite[Table 1]{Vera1}. Thus,   $(|V|-1)/|H|\geq 10,$which gives us $10\leq k_{p}(G)\leq 14$. If $k_{p'}(G)\geq 3$, then we have Theorem \ref{main}. Thus, we may assume that $k_{p'}(G)=2$, which gives us $|H|\leq 2$. Then we have the contradiction that  $(|V|-1)/|H|\geq 12194.$

Let $p=31$. We may assume that $3\leq k_p(G)\leq 15$ and so, $2\leq k_{p'}(G)\leq 10$. If $k_{p'}(G)\leq 9$, then $|H|\leq 2520$ by \cite[Table 1]{Vera1}. Thus,   $(|V|-1)/|H|\geq 12,$ which gives us $12\leq k_{p}(G)\leq 15$. If $k_{p'}(G)\geq 3$, then we have Theorem \ref{main}. Thus, we may assume that $k_{p'}(G)=2$, which gives us $|H|\leq 2$. Then we have the contradiction that  $(|V|-1)/|H|\geq 14895.$
It follows that we may assume that $k(H)=k_{p'}(G)=10$, which leads to $k_{p}(G)=3$. Thus, $|H|\leq 20160$ by \cite[Table 2]{Vera1} and also, $|V|=31^3$ because of $k_p(G)=3$. Since $k_p(G)=3$, the calculations show by \cite[Table 2]{Vera1}  that $H\cong (C_{11} \times C_{11})\rtimes \textrm{SL}(2, 5) $ or $\textrm{PSL}(3, 4)$. Since $H$ acts faithfully on $V$, we have that $H$ is a subgroup of $\textrm{GL}(3, 31)$, whose order is not divisible by the primes $7$ and $11$. On the other hand, either $11$ or $7$ divides the order of $H$, which is a contradiction. 

Let $p=37$. We may assume that $3\leq k_p(G)\leq 18$ and so, $2\leq k_{p'}(G)\leq 12$. If $k_{p'}(G)\leq 9$, then $|H|\leq 2520$ by \cite[Table 1]{Vera1}. Thus,   $(|V|-1)/|H|\geq 21,$ which is a contradiction. Thus, $10\leq k_{p'}(G)\leq 12$. This gives us $k_p(G)=3$. By examining  \cite[Tables 2, 3]{Vera1} and  \cite[Table 1]{Vera2}, we have that $H\cong \textrm{PSL}(3, 4)$, $\textrm{Sz}(8)$, $(C_{19} \times C_{19})\rtimes \textrm{SL}(2, 5)$ or $M_{22}$. Also, for these groups $H$ we can assume that $|V|=37^3$.  Since the order of  $\textrm{GL}(3, 37)$ is not divisible by $5$ and  $19^2$ we have the contradiction that $H$ is a subgroup of $\textrm{GL}(3, 37)$.

Let $p=41$. We may assume that $3\leq k_p(G)\leq 20$ and so, $2\leq k_{p'}(G)\leq 13$. If $k_{p'}(G)\leq 9$, then $|H|\leq 2520$ by \cite[Table 1]{Vera1}. Thus,   $(|V|-1)/|H|\geq 26,$ which is a contradiction. Thus, $10\leq k_{p'}(G)\leq 13$. This gives us $k_p(G)=3$ or $4$. By examining  \cite[Tables 2, 3]{Vera1},  \cite[Table 1]{Vera2}, \cite[Table 1]{Vera3} we have that $H\cong \textrm{PSL}(3, 4)$, $\textrm{Sz}(8)$, $(C_{19} \times C_{19})\rtimes \textrm{SL}(2, 5)$ or $M_{22}$.
Also,  we can assume that $|V|=41^3$.  Since the order of $\textrm{GL}(3, 41)$ is not divisible by  $9$, $11$, $13$ and  $19$ we have a contradiction that $H$ is a subgroup of $\textrm{GL}(3, 41)$.

Let $p=43$. We may assume that $3\leq k_p(G)\leq 21$ and so, $2\leq k_{p'}(G)\leq 14$. If $k_{p'}(G)\leq 9$, then $|H|\leq 2520$ by \cite[Table 1]{Vera1}. Thus,   $(|V|-1)/|H|\geq 32,$ which is a contradiction. Thus, $10\leq k_{p'}(G)\leq 14$. This gives us $k_p(G)=3$ or $4$. By examining  \cite[Tables 2, 3]{Vera1},  \cite[Table 1]{Vera2}, \cite[Table 1]{Vera3} we have that $H\cong \textrm{PSL}(3, 4)$, $\textrm{Sz}(8)$, $(C_{19} \times C_{19})\rtimes \textrm{SL}(2, 5)$, $M_{22}$, $\textrm{PSL}(3, 4) \cdot C_2$ or $\textrm{PSU}(3, 5)$.
Also, for these groups $H$ we can assume that $|V|=43^3$.  Since the order of $\textrm{GL}(3, 43)$ is not divisible by $5$ and  $19$ we have the contradiction that $H$ is a subgroup of $\textrm{GL}(3, 43)$.

We conclude that Theorem \ref{main} holds for every prime $p$ at most $43$.

\section{Almost quasisimple groups}

The purpose of this section is to prove the following.

\begin{proposition}
	\label{prop:k_star}
	Let $p$ be a prime at least $47$. Let $F$ be the finite field of order $q$ and characteristic $p$. Let $V$ be an absolutely irreducible, faithful and finite $FH$-module for a finite group $H$. Let $H=C\circ K$ where $C$ is a subgroup of the center $Z$ of $\GL(V)$ and $K$ is almost quasisimple. Suppose that $p$ does not divide $|H|$. Let $S$ be the socle of $K/Z(K)$. Let $|V|=p^n=q^d$ where $d=\dim_{F}(V)$. We have
	\begin{equation}
		\label{eq:V_over_p}
		|C|\cdot |S| \cdot |\mathrm{Out}(S)| \leq \frac{|V|}{q}
	\end{equation}
	except, possibly, for the cases indicated in Table \ref{possible_exceptions_V_over_p}. 
	
	Moreover, there exists a factorization $p-1=ab$ with $a$ and $b$ positive integers such that $k_{p'}(G)\geq k^{\ast}(S)\geq a$ and $k_p(G)=n(H,V)-1\geq b$, with equalities not occurring in both cases at the same time, except, possibly, for the cases indicated in Table \ref{possible_exceptions}. In particular, if $V$ is a primitive $FH$-module, then Theorem \ref{main} holds true for $G=HV$.
	
	\begin{table}[h!]
		\begin{center}
			\begin{tabular}{ |c|c|c|c| }
				\hline
				$d$ & $S$ & $q$ &  $q\cdot |H|/|V|<$ \\
				\hline
				$2$ & $A_5=\PSL_2(5)$ & all & - \\
				\hline
				$3$ & $A_5=\PSL_2(4)$ & $47,\dots,113$ & $2.5$\\
				\hline
				$3$ & $\PSL_2(7)$   &$47,\dots,331$  & $7$ \\
				\hline
				$4$ & $A_7$  & $47,\dots,67$ & $2.3$ \\
				\hline
				$6$ & $U_4(3)$ & $61,67$ & $1.9$ \\
				\hline
			\end{tabular}
		\end{center}
		\caption{Possible exceptions to \eqref{eq:V_over_p}.}
		\label{possible_exceptions_V_over_p}
	\end{table}	
    \vspace{-4mm}
	\begin{table}[h!]
		\begin{center}
			\begin{tabular}{ |c|c|c|c|c|c|c| }
				\hline
				$d$ & $S$ & $|\Out(S)|$ & $q$ & $k^{\ast}(S)$ & $(|V|-1)/|H|\geq $ & $|Z:C|$\\
				\hline
				$2$ & $A_5=\PSL_2(5)$ & $2$ & $\geq 47$ & $4$& $1$ & $\leq 60$\\
				\hline
				$3$ & $A_5=\PSL_2(4)$ & $2$  &$47$ & $4$ & $19$ & $1$\\
				\hline
				$3$ & $\PSL_2(7)$ & $2$  &$107$ & $5$ & $35$ & $1$\\
				$3$ & $\PSL_2(7)$ & $2$  &$103$ & $5$ & $32$ & $1$\\
				$3$ & $\PSL_2(7)$ & $2$  &$83$ & $5$ & $21$ & $1$\\
				$3$ & $\PSL_2(7)$ & $2$  &$79$ & $5$ & $19$ & $1$\\
				$3$ & $\PSL_2(7)$ & $2$  &$73$ & $5$ & $16$ & $1$\\
				$3$ & $\PSL_2(7)$ & $2$  &$67$ & $5$ & $14$ & $1$\\
				$3$ & $\PSL_2(7)$ & $2$  &$61$ & $5$ & $12$ & $1$\\
				$3$ & $\PSL_2(7)$ & $2$  &$59$ & $5$ & $11$ & $\leq 2$\\
				$3$ & $\PSL_2(7)$ & $2$  &$53$ & $5$ & $9$ & $1$\\
				$3$ & $\PSL_2(7)$ & $2$  &$47$ & $5$ & $7$ & $\leq 2$\\
				\hline
				$4$ & $A_7$ & $2$  &$47$ & $8$ & $22$ & $1$\\
				\hline
			\end{tabular}
		\end{center}
		\caption{Possible exceptions to Theorem \ref{main}.}
		\label{possible_exceptions}
	\end{table}	
    \vspace{-2mm}
\end{proposition}
\begin{proof}
	First notice that, since $|H|=|C\circ K|\leq |C|\cdot |S|\cdot |\Out(S)|$, it follows from \eqref{eq:V_over_p} that $|H|\leq |V|/q$. Hence, for groups $G=HV$ for which \eqref{eq:V_over_p} holds, our second claim holds with $a=1$ and $b=p-1$ since
	\begin{equation}
		\label{eq0:k_star}
		\frac{|V|}{|H|}-1
		<\frac{|V|-1}{|H|}
		\leq n(H,V)-1
		=k_p(G).
	\end{equation}
	If \eqref{eq:V_over_p} does not hold, we aim at showing that $|V|/|H|\geq (p+1)/2$ which allows us to take $a=2$ and $b=(p-1)/2$ in the proposition. For the possible exceptions, we calculate $|V|/|H|$ in order to find the constant $c$ in the last column of Table \ref{possible_exceptions_V_over_p} and we look for a factorization $p-1=ab$ with $k^{\ast}(S)\geq a>c$ and $b\geq (p-1)/c$. Notice that, we may assume $k^{\ast}(S)\geq 4$ and $k_p(G)=n(H,V)-1\geq 4$. We proceed as in the proof of \cite[Proposition 5.1]{Attila2016} and check the claim with \cite{HM2001,KL}. At various steps we need to determine certain thresholds for certain inequalities to hold. While these can be checked by hand, we used GAP \cite{GAP} for such calculations.
	
	We check \eqref{eq:V_over_p} by checking 
	\begin{equation}
		\label{eq1:V_over_p}
		|S|\cdot|\Out(S)|\leq q^{d-2}
	\end{equation}
	since then
	$$
	|C|\cdot|S|\cdot|\Out(S)|
	\leq (q-1)\cdot|S|\cdot|\Out(S)|
	< q\cdot|S|\cdot|\Out(S)|
	\leq q^{d-1}
	\leq \frac{|V|}{q}.
	$$
	
	When \eqref{eq1:V_over_p} fails, we check
	\begin{equation}
		\label{eq1:k_star}
		|S|\cdot|\Out(S)|< 1.94 \frac{q^{d-1}}{p}
	\end{equation}
	since then
	$$
	|H|< 1.94\frac{|V|}{p}\leq 2\frac{|V|}{p+1}
	$$
	which implies $(p-1)/2< k_p(G)$.
	
	First we deal with \cite[Table 3]{HM2001}. Let $S$ and $d$ be as in \cite[Table 2]{HM2001correction}. Since $p\geq 47$ and $|\Out(S)| < |S|$, a calculation shows that \eqref{eq1:V_over_p} holds true for $d\geq 25$ since
	$$
	|S|\cdot|\Out(S)|
	<|S|^2
	\leq 47^{d-2}
	\leq q^{d-2}.
	$$
	Let $d\leq 24$. Using the exact values $|\Out(S)|$ one checks that \eqref{eq1:V_over_p} holds true except possibly for $(d,S)=(4,A_7)$ or $(d,S)=(6,U_4(3))$. 
	
	Let $(d,S)=(4,A_7)$. By varying $q$, a calculation shows that \eqref{eq1:V_over_p} holds except if $q\leq 67$ and we indicate an upper bound on $q\cdot |H|/|V|$ in Table \ref{possible_exceptions_V_over_p}. Moreover, in this case $|\Out(S)|=2$ and one checks that \eqref{eq1:k_star} holds true except possibly for $p=q=47$. For this value of $q$, using \eqref{eq0:k_star} we see that, if $C=1$ we have $k_p(G)\geq (|V|-1)/|H| \geq 967.1$ and if $C=Z$ we have $k_p(G)\geq (|V|-1)/|H|\geq 21.04$.
	
	Let $(d,S)=(6,U_4(3))$. By varying $q$, a calculation shows that \eqref{eq1:V_over_p} holds except if $q\leq 71$. However, since $|\Out(S)|=8$, it follows from \cite[Table 2]{HM2001correction} that $3$ divides $|C|$ which cannot happen except possibly if $q\in\{61,67\}$ and we indicate an upper bound on $q\cdot |H|/|V|$ in Table \ref{possible_exceptions_V_over_p}. Moreover, using $|\Out(S)|=8$, one checks that \eqref{eq1:k_star} holds true if $q>p$ or if $q\geq 61$. 
	
	Let $S=A_n$ with $n>5$. The case $A_5=\PSL_2(4)=\PSL_2(5)$ is treated below. Since $p$ does not divide $|G|$, we have $p\geq n+1$ and, by \cite[Proposition 5.3.7]{KL}, we have $d\geq n-2$ for $n\geq 9$. The entries for $A_n$ with $n\geq 14$ have been omitted from \cite[Table 2]{HM2001correction} (see beginning of Section 6 in \cite{HM2001}). Thus, we assume that either $n\geq 14$ or $5< n\leq 13$ and $d\geq 251$. Let $x=n-2$ in the first case and let $x=251$ in the second case. In both cases, inequality \eqref{eq1:V_over_p} holds true since
	$$
	|S|
	\leq \frac{(n+1)^{x-2}}{4}
	\leq \frac{p^{d-2}}{4}
	\leq \frac{q^{d-2}}{|\Out(S)|}.
	$$
	
	Let $S=\PSL_2(f)$. Let $f$ be even. Then $|S|=f(f^2-1)$ and $|\Out(S)|\leq f$. By \cite[Table 5.3.A]{KL}, we have $d\geq f-1$ except possibly if $f=4$. However, by \cite[Table 2]{HM2001}, this exception does not occur. A calculation shows that \eqref{eq1:V_over_p} holds true for $f\geq 8$ since
	$$
	|S|\cdot |\Out(S)|\leq 47^{f-3}.
	$$
	If $f=4$, then $S=\PSL_2(4)=A_5$ and $d=3$. A calculation shows that \eqref{eq1:V_over_p} holds true for $q\geq 127$. For $q\leq 113$ we indicate an upper bound on $q\cdot |H|/|V|$ in Table \ref{possible_exceptions_V_over_p}. Moreover, one checks that for $q>p$ we have $p<|V|/|H|\leq k_p(G)$ and that
	$$
	\frac{p+1}{2}\leq \frac{q^3}{\frac{p-1}{2}\cdot|\Out(S)|\cdot|S|}
	$$
	for all $q\geq 47$. This proves our claim when $C\neq Z$. A similar calculation shows that \eqref{eq1:k_star} may not hold for $p=q\in\{47,53,59,61\}$ in which case $|V|/|H|> 18.8$, $23.8$, $29.5$ and $51.5$ respectively. Thus, since $k^{\ast}(S)=4$, our claim follows for $p\in\{53,59,61\}$ using \eqref{eq0:k_star}.
	
	Let $f$ be odd. Then $|S|=f(f^2-1)/2$ and $|\Out(S)|\leq 2\cdot f$. By \cite[Table 5.3.A]{KL}, we have $d\geq (f-1)/2$ except possibly if $f=9$. However, by \cite[Table 2]{HM2001}, this exception does not occur. A calculation shows that \eqref{eq1:V_over_p} holds true for $f\geq 11$. The remaining cases are $f\in\{5,7,9\}$. By \cite[Table 2]{HM2001}, we have $d\geq 2,3,4$ respectively. 
	
	Let $f=9$ or $7$. A calculation shows that $q<|V|/|H|$ for $f=9$ and that \eqref{eq1:V_over_p} holds true for $f=7$ if $q\geq 337$. 
	For $q\leq 331$ we indicate an upper bound on $q\cdot |H|/|V|$ in Table \ref{possible_exceptions_V_over_p}.
	Calculating further, we find that $p<|V|/|H|\leq k_p(G)$ for $q>p$ and that \eqref{eq1:k_star} holds true since
	$$
	\frac{p+1}{2}<\frac{q^3}{(q-1)\cdot|\Out(S)|\cdot|S|}
	$$
	except if $p=q\leq 167$. For the remaining cases, when $S=\PSL_2(7)$, we calculate $(|V|-1)/|H|$ explicitly. Since $k^{\ast}(S)=5$, using \eqref{eq0:k_star}, we may rule out several of these cases. The remaining cases are $47\leq p=q\leq 107$. If $C\neq Z$, the claim holds except possibly if $p=47$ or $p=59$. In these two cases $|Z|=2\cdot 23$ and $|Z|=2\cdot 29$ respectively and one checks that our claim holds if $|Z:C|>2$.
	
	If $f=5$ then $S=\PSL_2(5)=A_5$ and $d=2$. A calculation shows that  \eqref{eq:V_over_p} holds true if $q\geq p^2\geq 121$. A similar calculation with \eqref{eq0:k_star} shows that \eqref{eq1:k_star} holds true for $p\geq 61$. If $p=q$ then $d=n=2$ and if we assume that $|Z:C|>60=|S|$, then $k_p(G)\geq (p-1)/2$ since 
	$$
	\frac{p+1}{2}
	<\frac{p^2-1}{|C|\cdot 2\cdot 60}
	<\frac{p^2}{|C|\cdot 2\cdot 60}
	=\frac{p^2}{|C|\cdot |\Out(S)|\cdot |S|}.
	$$
	
	For the rest of the proof we exclude the cases appearing in the tables of \cite{HM2001,HM2001correction}. In particular, we assume $d\geq 251$ and we may ignore the exceptions listed in \cite[Table 5.3.A]{KL}. Furthermore, with \cite[Proposition 5.3.8]{KL} and since $d\geq 251$, we check that \eqref{eq1:V_over_p} holds true for the sporadic groups and the Tits group. In what follows we treat the remaining simple groups of Lie type $S$ for which we bound $d$ with \cite[Theorem 5.3.9]{KL}.
	
	Let $S={}^2G_2(f)$ with $f=3^{2m+1}$ and $m\geq 1$. Here $d\geq f(f-1)$ and a calculation shows that \eqref{eq1:V_over_p} holds true for all $m$ since
	$$
	|S|\cdot \log_3(f)\leq 47^{f(f-1)-2}.
	$$
	
	Let $S=Sz(f)={}^2B_2(f)$. Here $f=2^{2m+1}$ with $m\geq 1$ and $d\geq \sqrt{f/2}\cdot(f-1)$. A calculation shows that \eqref{eq1:V_over_p} holds true for all $m$ since
	$$
	|S|\cdot\log_2(f)\leq 47^{\sqrt{f/2}\cdot(f-1)-2}.
	$$
	
	For the rest of the groups it is computationally more convenient to take the logarithm of \eqref{eq1:V_over_p}. 
	Notice that $|S|\leq (f+1)^{\dim S}$ where $\dim(S)$ is the dimension of the ambient algebraic group. Then, for \eqref{eq1:V_over_p}, it suffices to show that
	$$
	\log|S|\leq
	\log  ((f+1)^{\dim S})
	\leq \log\frac{47^{d_{\min}-2}}{|\Out(S)|}
	\leq \log\frac{q^{d-2}}{|\Out(S)|}
	$$
where $d_{\min}$ is at least $251$ and at least the lower bound on $d$ given in \cite[Theorem 5.3.9]{KL}. Thus it suffices to check the values $f$ for which
\begin{equation}
	\label{eq2:k_star}
	\dim(S)\cdot\log(f+1) 
	\leq
	5\cdot(d_{\min}-2)-\log(|\Out(S)|).
\end{equation}

Let $S={}^2F_4(f)$. Here $f=2^{2m+1}$ with $m\geq 1$. We have $|S|<f^{26}$ and $|\Out(S)|=2m+1$. A calculation shows that \eqref{eq2:k_star} holds true for all $f$ since
$$
26\cdot (2m+1)
<5\cdot\left(f^4\cdot\sqrt{\frac{f}{2}}\cdot(f-1)-2\right)-\log(2m+1).
$$

Let $S={}^3D_4(f)$. We have $\dim(S)=28$ and $|\Out(S)|\leq 3\cdot \log(f)$. A calculation shows that \eqref{eq2:k_star} holds true for all $f\neq 2$ since
$$
28\cdot\log(f+1)
<5\cdot\big(f^3\cdot(f^2-1)-2\big)-\log(3\cdot \log(f)).
$$
For $f=2$ we use $d\geq 251$. Similar calculations show that \eqref{eq2:k_star} holds true except possibly if $S$ is one of the groups $G_2(2)$, $G_2(3)$. For these two cases, we check that \eqref{eq2:k_star} holds true under the assumption that $d\geq 251$.

For unbounded rank we use the fact that the lower bounds on $d$ in \cite[Theorem 5.3.9]{KL} are bounded from below by $f^{d'}$ for some integer $d'>0$. Observe that $f/\log(f+1)\geq 1$ for all $f$. We divide \eqref{eq2:k_star} by $\log(f+1)$ and notice that it suffices to show that
\begin{equation}
	\label{eq4:k_star}
	\dim(S)+c_{\Out}+10 
	\leq
	5\cdot f_{\min}^{d'-1},
\end{equation}
where $f_{\min}$ is the smallest possible value for $f$ and where $c_{\Out}$ is a constant such that $\log(|\Out(S)|)\leq c_{\Out}\cdot \log(f+1)$.

Let $S=\Omega_{2m+1}(f)$ with $m\geq3$ and $f$ odd. We have $\dim(S)=2m^2+m$ and $|\Out(S)|\leq 2\cdot \log(f)$. One checks that we may take $c_{\Out}=3$, that $d'=2m-3$ and $f_{\min}=3$. Since $m\geq 3$, a calculation shows that \eqref{eq4:k_star} holds true for all $m$ as
$$
2m^2+m+3+10
\leq 5\cdot 3^{2m-4}.
$$

Let $S=\Omega_{2m}^{-}(f)$ with $m\geq4$. We have $\dim(S)=2m^2-m$ and $|\Out(S)|\leq 8\cdot \log(f)$. One checks that we may take $c_{\Out}=3$, that $d'=2m-4$ and $f_{\min}=2$. A calculation shows that \eqref{eq4:k_star} holds true for all $m\geq 5$. We have
$$
2m^2-m+3+10
\leq 5\cdot 2^{2m-5}.
$$

Let $S=\Omega_{8}^{-}(f)$. We have $\dim(S)=28$ and $|\Out(S)|\leq 8\cdot\log(f)$. A calculation shows that \eqref{eq2:k_star} holds true for all $f>2$ since
$$
28\cdot\log(f)+1
<5\cdot\big((f^{3}+f)\cdot(f^2-1)-2\big)-\log(8\cdot\log(f)).
$$
For $f=2$ we use $d\geq 251$.

Let $S=\Omega_{2m}^{+}(f)$ with $m\geq4$. We have $\dim(S)=2m^2-m$ and $|\Out(S)|\leq 6\cdot \log(f)$. One checks that we may take $c_{\Out}=3$ that $d'=2m-4$ and $f_{\min}=2$. A calculation shows that \eqref{eq4:k_star} holds true for all $m\geq 5$ since we have
$$
2m^2-m+3+10
\leq 5\cdot 2^{2m-5}.
$$

Let $S=\Omega_{8}^{+}(f)$. We have $|S|<f^{28}$ and $|\Out(S)|\leq 6\cdot\log(f)$. A calculation shows that \eqref{eq2:k_star} holds true for all $f>2$ since
$$
28\cdot\log(f)
<5\cdot\big((f^{3}+f)\cdot(f^2-1)-2\big)-\log(8\cdot\log(f)).
$$
For $f=2$ we use $d\geq 251$.

Let $S=PSp_{2m}(f)$ with $m\geq2$. We have $\dim(S)=2m^2+m$ and $|\Out(S)|\leq 2\cdot \log(f)$. One checks that we may take $c_{\Out}=3$ and $f_{\min}=3$, that $d'=m-1$ if $f$ is odd and that $d'=2m-4$ if $f$ is even. A calculation shows that 
$$
2m^2+m+3+10
\leq 5\cdot 2^{d'-1}.
$$

Let $S=PSp_{2m}(f)$ with $m\leq 5$. A calculation shows that \eqref{eq2:k_star} holds true except possibly if $(m,f)\in\{(2,2),(2,3),(2,4),(2,5),(3,2),(3,3),(4,2),(4,3)\}$. For these cases, we check that \eqref{eq2:k_star} holds true under the assumption that $d\geq 251$.

Let $S=L_{m}(f)$ with $m\geq3$. We have $\dim(S)=m^2-1$ and $|\Out(S)|\leq 2\cdot m\cdot \log(f)$. One checks that we may take $c_{\Out}=2m$, that $d'=m-2$ and $f_{\min}=2$. A calculation shows that \eqref{eq4:k_star} holds true for all $m\geq 7$ since we have
$$
m^2-1+2m+10
\leq 5\cdot 2^{2m-5}.
$$

Let $S=L_{6}(f)$ with $m\leq 5$. A calculation shows that \eqref{eq2:k_star} holds true except possibly for $(m,f)\in\{(3,3),(3,4),(4,3)\}$. For these cases, we check that \eqref{eq2:k_star} holds true under the assumption that $d\geq 251$.

Let $S=U_{m}(f)$ with $m\geq3$. We have $\dim(S)=m^2-1$ and $|\Out(S)|\leq 2\cdot m\cdot \log(f)$. One checks that we may take $c_{\Out}=2m$, $d'=m-2$ and $f_{\min}=2$. The calculation to see that \eqref{eq4:k_star} holds is the same as for the case $S=L_{m}(f)$.

In order to prove the last claim of the proposition it is sufficient to consider the cases in Table \ref{possible_exceptions}. For the cases in the table we use the fact that $k_{p'}(G)=k(H)\geq |C|\cdot k^{\ast}(S)$ which in all cases, apart from the case when $n=d=2$, is at least $p-1$. The case $n=d=2$ was treated in Step 4 of Section 6.
\end{proof}

Next we consider certain imprimitive modules. 

\begin{proposition}
\label{prop:imprimitive_case}
Let $p$ be a prime at least $47$. Let $V$ be an irreducible and imprimitive $FH$-module for a finite field $F$ of characteristic $p$ and a finite group $H$. Let $V$ be induced from an $FL$-module $W$ for a subgroup $L$ of $H$ of index $t$. Let $A$ be the kernel of the action of $L$ on $W$. Assume that $L/A=C\circ K$, where $K/Z(K)$ is almost simple with socle $S$ and the size of $K$ is not divisible by $p$ and where $C$ is a subgroup of the center $Z$ of $\GL(W)$ such that $W$ is an absolutely irreducible $F(L/A)$-module. If $\dim_{F}(W)\geq 3$, then Theorem \ref{main} holds true for $G=HV$. 
\end{proposition}

\begin{proof}
We may suppose that $t \geq 2$ by Proposition \ref{prop:k_star}. Since $H$ is not solvable, $k_{p'}(G) \geq 3$ by Burnside's theorem. Consider the statement of Proposition \ref{prop:k_star} with $H$ replaced by $L/A$, $V$ replaced by $W$ and $G$ replaced by $(L/A)W$. Let $k=n(L/A,W)$. Observe that $$k_p(G) \geq \binom{k+t-1}{k-1}-1 \geq \frac{k(k+1)}{2}-1 \geq k-1.$$
We have $k \geq q \geq p$ by Proposition \ref{prop:k_star} unless $L/A$, $W$ and $S$ are as in Table \ref{possible_exceptions_V_over_p} (with $\dim_{F}(W)\geq 3$). In the exceptional cases we have $k \geq q/7 \geq p/7$ and so $k(k+1)/2-1\geq (p-1)/2$ for $p \geq 47$.
\end{proof}

\section{Metacyclic sections}





The purpose of this section is to prove Theorem \ref{main} in the special case when the $H$-module $V$ is induced from a subspace $W$ such that the stabilizer of $W$ in $H$ modulo the kernel is a metacyclic group that acts irreducibly on $W$.

We start with a lemma that will be used not only in this section but also in a later part of the paper. 

\begin{lemma}
\label{newlemma}	
If $G = HV$ is a counterexample to Theorem \ref{main} with a prime $p$, then $H$ has no alternating composition factor of degree at least $(\ln(112)+\ln p)^2/4$.
\end{lemma}

\begin{proof}
Assume that $S = A_n$ is an alternating composition factor of $H$ with $n \geq (\ln(112)+\ln p)^2/4$. Since $|\mathrm{Out(S)}| \leq 4$, we have $k^{*}(S) \geq k(S)/4$. Since $S$ is a normal subgroup of index $2$ in $S_n$, we have $k(S) \geq \pi(n)/2$ where $\pi(n)$ denotes the number of partitions of $n$. We have $k(H) \geq k^{*}(S)\geq 4$ by \cite[Lemma 2.5]{pyber1992}. Thus, by Lemma \ref{l1} and by \cite[Corollary 3.1]{Maroti2003} we have
$$
k_{p'}(G) =k(H)\geq \max\{4,\frac{\pi(n)}{8}\}\geq \max\{4,\frac{e^{2\sqrt{n}}}{112}\}\geq p
$$
where the last inequality holds provided that $n \geq (\ln(112)+\ln p)^2/4$.
\end{proof}

Let an imprimitivity decomposition of the irreducible $H$-module $V$ be $V_{1} + \cdots + V_{t}$. For each $i$ with $1 \leq i \leq t$, the vector space $V_i$ is a primitive $H_i$-module where $H_i$ is the stabilizer of $V_i$ in $H$. The group $H$ acts transitively on the set $\{ V_{1}, \ldots , V_{t} \}$. Let the kernel of this action be $B$. The factor group $H/B$ may be considered as a transitive permutation group of degree $t$. Let $m$ denote the minimal degree of a non-abelian alternating composition factor of $H/B$, provided that such exists, otherwise $m=4$. We have $|H/B| \leq m!^{(t-1)/(m-1)}$ by \cite[Corollary 1.5]{M1}. 


The group $B$ may be considered as a subgroup of $B_{1} \times \cdots \times B_{t}$ for isomorphic groups $B_{1}, \ldots , B_{t}$ such that $B$ projects onto each factor $B_i$. Moreover, for each $i$ with $1 \leq i \leq t$, the group $B_i$ may be considered as a normal subgroup in a primitive linear group acting on $V_i$. 

For each $i$ with $1 \leq i \leq t$, let $A_i$ be a largest abelian normal subgroup in $B_i$. Let the index of $A_i$ in $B_i$ be $f$. Let $C$ be the abelian normal subgroup of $B$ consisting of elements $(b_{1}, \ldots , b_{t}) \in B_{1} \times \cdots \times B_{t}$ with the property that $b_{i} \in A_{i}$ for all $i$ with $1 \leq i \leq t$. We claim that $|B:C| \leq f^{t}$. Let $E = B_{1} \times \cdots \times B_{t}$ and let $D = A_{1} \times \cdots \times A_{t}$. Observe that $D \cap B = C$ and $D$ is normalized by $B$. Now $B/C =  B/(D \cap B) \cong DB/D \leq E/D$. But $|E/D| = f^{t}$.  

\begin{lemma}
\label{meta}	
Let $G = HV$ be a counterexample to Theorem \ref{main} with a prime $p$. Fix $i$ with $1 \leq i \leq t$. The group $H_{i}/C_{H_{i}}(V_{i})$ is not a subgroup of $\Gamma L(1,K) \leq \mathrm{GL}(n/t,F)$ for any field extension $K$ of the prime field $F$ of order $p$.
\end{lemma}

\begin{proof}
Fix $i$ with $1 \leq i \leq t$. Assume that $H_{i}/C_{H_{i}}(V_{i})$ is metacyclic. Then $B_i$ is metacyclic. We have $|B_{i} : A_{i}| = f$ and $n \geq t f$ (since we view $V$ and $V_i$ over the field of size $p$). The index of the abelian subgroup $C$ in $H$ satisfies $$|H:C| \leq |H/B||B/C| \leq m!^{(t-1)/(m-1)} \cdot f^{t} \leq m!^{(t-1)/(m-1)} \cdot (n/t)^{t}.$$ 

Since $H$ has less than $p$ orbits on $V$, it follows that
$$
p^{n-1} \leq |H| \leq m!^{(t-1)/(m-1)} \cdot (n/t)^{t} \cdot |C|.
$$

We have $k_{p'}(G) = k(H) \geq k^{*}(A_{m}) \geq \max\{4,\pi(m)/8\}$ by the proof of Lemma \ref{newlemma}. If this is at least $p/2$, then there is nothing to show. Note that we may assume that $p/2 \leq 60^{4}/2$. Now $\max\{4,\pi(m)/8\} \geq 60^{4}/2$ provided that $m \geq 90$. We may thus assume that $m \leq 89$ (and $|H/B| < 36^{t-1}$).  

For $m\geq 5$, we have
\begin{equation}
  \label{eq11}
  k_{p'}(G) = k(H) \geq \frac{|C|}{|H:C|} + k^{*}(A_{m}) \geq \frac{p^{n-1}}{m!^{2(t-1)/(m-1)} \cdot (n/t)^{2t}} + k^{*}(A_{m}).
\end{equation}
Since $k^{*}(A_{m})\geq \max\{4,\pi(m)/8\}$, we have
\begin{equation}
  \label{eq12}
  k_{p'}(G) \geq \frac{p^{n-1}}{m!^{2(t-1)/(m-1)}\cdot (n/t)^{2t}}+\max\{4,\pi(m)/8\}.
\end{equation}
For $m=4$, we have 
\begin{equation}
  \label{eq13}
  k_{p'}(G) = k(H) \geq \frac{|C|}{|H:C|} \geq \frac{p^{n-1}}{24^{2(t-1)/3} \cdot (n/t)^{2t}}.
\end{equation}
Notice that, since $t\mid n$, we have $(n/t)^{2t}\leq 3^{2n/3}$. Indeed, let $a=n/t$ then the inequality follows from $a^{1/a}\leq 3^{1/3}$ (since $a^{1/a}$ decreases for $a\geq 4$). Notice also that we may assume $k_p(G) > 3$ (by \cite[Theorem 1.1 and Section 2]{HungSambaleTiep}). We may assume $p\geq 47$.


Let $n \geq 2t$. Let $5\leq m\leq 89$. Since $n\geq m\geq 5$, a calculation shows that
$$
k_{p'}(G)\geq\frac{p^{n-1}}{89!^{2(t-1)/88} \cdot 3^{2n/3}}> \frac{p^{n-1}\cdot 35^2}{74^n}  \geq \frac{p-1}{2}
$$
for $p\geq 97$, where the first inequality holds since $m\leq 89$. Let $p\leq 89$. By a GAP \cite{GAP} calculation, we have $k^{\ast}(A_m)> 44\geq p-1$ for $m\geq 13$. By \eqref{eq11}, we may assume that $m\leq 12$. Then, \eqref{eq11} gives
$$
k_{p'}(G) \geq \frac{p^{n-1}}{12!^{2(t-1)/11}\cdot 3^{2n/3}}> \frac{p^{n-1}\cdot 6^2}{13^n}\geq\frac{p-1}{2}
$$
for $p\geq 17$.

Let $n=t$. Let $5 \leq m \leq 89$. Observe that $n \geq m$. Thus, \eqref{eq12} and a GAP \cite{GAP} calculation give
$$
k_{p'}(G) \geq \frac{p^{m-1}}{m!^{2}} + \max\{4,\pi(m)/8\}> \frac{p-1}{2}
$$
for $p \geq 67$ or $m\geq 16$. Let $p \leq 61$ and $m\leq 15$. Then, by \eqref{eq11} we have
$$
k_{p'}(G) \geq   \frac{p^{n-1}}{m!^{2}} + k^{*}(A_{m})> \frac{p-1}{2}
$$
for all $m$ and $p\geq 47$.
\end{proof}

\section{Groups with few orbits}

Let $p$ be a prime at least $47$. Let $F$ be the field of order $q$ and characteristic $p$. Let $V$ be an absolutely irreducible, primitive and faithful $FH$-module for a finite group $H$ of order coprime to $p$. Let $|V| = p^{n} = q^{d}$ where $d = \dim_{F}(V)$. We proceed to describe the cases when $|H| > |V|/p$. Otherwise, $p \leq |V|/|H| \leq n(H,V)$.  

In the first step, assume that every irreducible $N$-submodule of $V$ is absolutely irreducible for any normal subgroup $N$ of $H$. Let $H$ be different from a cyclic group. We follow the proof of \cite[Theorem 4.1]{GMP} with $H:= A = G$ in that notation. Let $J_1, \ldots , J_k$ denote the distinct normal subgroups of $H$ that are minimal with respect to being noncentral in $H$. Let $J = J_{1} \cdots J_{k}$ be the central product of these subgroups. The group $H/(Z(H)J)$ embeds into the direct product of the outer automorphism groups $O_i$ of the $J_i$. Let $W$ be an irreducible constituent for $J$. We have $W \cong U_{1} \otimes \cdots \otimes U_{k}$ where $U_i$ is an irreducible $J_i$-module. If $J_i$ is the central product of $t$ copies of a quasisimple group $Q$, then $\dim U_{i} \geq m^{t}$ where $m$ is the dimension of the nontrivial module for $Q$ whose $t$-th tensor power is $U_i$. If $J_{i}$ is a group of symplectic type with $J_{i}/Z(J_{i})$ of order $r^{2a}$ for a prime $r$ and an integer $a$ then $\dim U_{i} = r^{a}$. 

If a subgroup $J_i$ is quasisimple, then $|Z(H) J_{i}||O_{i}| < 3 |U_{i}|$ by Proposition \ref{prop:k_star}. Moreover, if $d_{i} = \dim(U_{i})$ is different from $2$, then $|Z(H) J_{i}||O_{i}| < |U_{i}|/6$ and if $d_{i} \geq 7$ then $|Z(H) J_{i}||O_{i}| < |U_{i}|/q$. If $J_i$ is the central product of $t$ copies of a quasisimple group $Q$ and $m$ is the dimension of the nontrivial module for $Q$ whose $t$-th tensor power is $U_i$, then $|Z(H) J_{i}||O_{i}| < 3^{t} q^{tm} p^{t-1}$ by Proposition \ref{prop:k_star} and by \cite{pp}. It follows that $|Z(H)J_{i}||O_{i}| < |U_{i}|/q$ for $t \geq 4$. Let $m = 2$. If $t = 2$, then $|Z(H)J_{i}||O_{i}| < 2 (q-1) 120^{2} < |U_{i}|$. If $t = 3$, then $|Z(H)J_{i}||O_{i}| < 6 (q-1) 120^{3} < |U_{i}|$. If $m \geq 3$, then $|Z(H)J_{i}||O_{i}| < q^{tm}/6^{t} < |U_{i}|/q$. In general, if $d_{i} \geq 9$, then $|Z(H)J_{i}||O_{i}| < |U_{i}|/q$.    

If $J_{i}$ is a group of symplectic type with $J_{i}/Z(J_{i})$ of order $r^{2a}$ for a prime $r$ and an integer $a$, then $|Z(H)J_{i}||O_{i}| < |U_{i}|/q$, unless $(r,a) \in \{ (2,1), (3,1), (2,2) \}$ by the first two paragraphs of the proof of \cite[Proposition 5.2]{Attila2016}. In particular, if $d_{i} \geq 5$, then $|Z(H)J_{i}||O_{i}| < |U_{i}|/q$, and $|Z(H)J_{i}||O_{i}| < |U_{i}|$ in all cases. Moreover, $|Z(H)J_{i}||O_{i}| < |U_{i}|/3$ for $d_{i} = 3$. 

We have $|H| \leq \prod_{i=1}^{k} |Z(H)J_{i}||O_{i}|$ and so $$|H| \leq 3^{k} \prod_{i=1}^{k} |U_{i}| = 3^{k} q^{\sum_{i=1}^{k}d_{i}}.$$ Assume, without loss of generality, that $2 \leq d_{1} \leq \cdots \leq d_{k}$. We may write $|H| \leq 3^{k} q^{k d_{k}}$. This is less than $q^{2^{k-1}d_{k} - 1} \leq q^{(\prod_{i=1}^{k} d_{i}) -1} = |U_{i}|/q$ provided that $k \geq 3$. If $k =1$ and $d \geq 9$, then $|H| < |V|/q$ from the above. Let $k=2$. We have $$|H| \leq 9 \cdot  q^{d_{1} + d_{2}} < q^{d_{1}d_{2}-1} = |W|/q \leq |V|/q$$ provided that $d \geq 8$. Since both $d_{1}$ and $d_{2}$ are at least $2$, the remaining cases are $d_{1}=2$ and $d_{2} \in \{ 2, 3 \}$. If $d_{1} = 2$ and $d_{2} = 3$, then $|H| < |U_{1}||U_{2}| = q^{5} \leq |V|/q$. Let $d_{1} = d_{2} = 2$. If both $J_1$ and $J_2$ are of symplectic type, then $|H| \leq 24^2 (q-1) < |V|/q$. Let $J_1$ be nonsolvable. If $q > 120$, then $|H| \leq 120^2 (q-1) < |V|/q$. We are left with the primes $59$, $61$, $71$, $79$, $89$, $101$, $109$ by \cite[II, Hauptsatz 8.27]{huppert}. If $J_2$ is of symplectic type, then $|H| \leq 120 \cdot 24 \cdot (q-1) < |V|/q$. Let $J_2$ be nonsolvable. 

Let us conclude our finding. Assume that every irreducible $N$-submodule of $V$ is absolutely irreducible for any normal subgroup $N$ of $H$. Let $H$ be different from a cyclic group. Let $p \geq 47$. Let $H$ have order not divisible by $p$. Then $|H| < |V|/q$ unless possibly if $J = J_{1}$ and $H$ is as in Table \ref{possible_exceptions_V_over_p} (with $q=p$ unless $d_{1} =2$) or $J = J_1$ is of symplectic type and $d_{1}$ is $2$, $3$ or $4$, or $k=2$, $d_{1} = d_{2} = 2$, both $J_1$ and $J_2$ are nonsolvable, and $q = p \in \{  59, 61, 71, 79, 89, 101, 109 \}$.  

We follow the proof of \cite[Theorem 4.2]{GMP} with $H:= A = G$. From the previous paragraph, we find that if $V$ is a primitive and faithful $H$-module, $p \geq 47$ and $H$ has order not divisible by $p$, then $H \leq \mathrm{\Gamma L}(1,Q)$ for some field extension $Q$ of $F$, or $H$ is almost quasisimple as in Table \ref{possible_exceptions_V_over_p} (with $\dim V \geq 3$), or $|H| \leq |V|/p$, or $H$ has two normal subgroups $J_1$ and $J_2$ which are minimal with respect to being noncentral in $H$, both $J_1$ and $J_2$ are as in the first row of Table \ref{possible_exceptions_V_over_p} and $q = p \in \{  59, 61, 71, 79, 89, 101, 109 \}$, or there is a divisor $e$ of $d$ such that $H$ contains a normal subgroup $L = H \cap \mathrm{GL}(d/e,q^{e})$ of index $e$ in $H$ acting primitively and irreducibly on a vector space $U$ over the extension field of $F$ of order $q^e$ such that (i) $L$ is almost quasisimple as in the first row of Table \ref{possible_exceptions_V_over_p} with $d/e$ in place of $d$, $q^{e}$ in place of $q$, $U$ in place of $V$ and $L$ in place of $H$ or (ii) $L$ has a unique normal subgroup $J$ which is minimal with respect to being noncentral in $L$, $J$ is of symplectic type, it acts absolutely irreducibly on $U$ and $d/e \in \{ 2, 3, 4 \}$. 

\section{Proof of Theorem \ref{main}}

In this section we finish the proof of Theorem \ref{main}. 

\medskip

{\bf Setup.} We may suppose by Section 6 that $G = HV$ where $V$ is an elementary abelian normal $p$-subgroup in $G$ and $H$ is a subgroup of $G$ of order not divisible by $p$. Moreover, $V$ is a faithful and irreducible $FH$-module of order at least $p^{3}$ for a finite field $F$ of characteristic $p$. We may suppose that $p \geq 47$ by Section 7. Let $q = |F|$ and let $|V| = q^{d}$ for some integer $d$.  

As before, let an imprimitivity decomposition of the irreducible $H$-module $V$ be $V_{1} + \cdots + V_{t}$ with $t \geq 1$. For each $i$ with $1 \leq i \leq t$, the vector space $V_i$ is a primitive (and irreducible) $H_i$-module where $H_i$ is the stabilizer of $V_i$ in $H$. The group $H$ acts transitively on the set $\{ V_{1}, \ldots , V_{t} \}$. Let the kernel of this action be $B$. The factor group $H/B$ may be considered as a transitive permutation group of degree $t$. Let $m$ denote the minimal degree of a non-abelian alternating composition factor of $H/B$, provided that such exists, otherwise $m=4$. We have $|H/B| \leq m!^{(t-1)/(m-1)}$ by \cite[Corollary 1.5]{M1}. Moreover, $m < {(\ln(112) + \ln p)}^{2}/4$ by Lemma \ref{newlemma}.

Put $W = V_1$ and let $K = H_1$. The index of $K$ in $H$ is $t$. We will also suppose that $k_{p}(G) \geq 3$ by the beginning of Section 7. 

\medskip

{\bf Using Section 10.} In order to prove Theorem \ref{main}, we use Section 10, or more precisely, the last paragraph of Section 10 to collect information about the group $K/C_{K}(W)$.

If $K/C_{K}(W) \leq \mathrm{\Gamma L}(1,Q)$ for some field extension $Q$ of the underlying field $F$, then Theorem \ref{main} holds by Lemma \ref{meta}. If $K/C_{K}(W)$ is almost quasisimple as in Table \ref{possible_exceptions_V_over_p} with $\dim W \geq 3$, then Theorem \ref{main} holds by Proposition \ref{prop:imprimitive_case}. If $|K/C_{K}(W)| \leq |W|/p$ (and $K/C_{K}(W) \not= 1$), then the number of nontrivial orbits of $K$ on $W$ is at least $p-1$ and so $k_{p}(G) \geq p-1$ and $k_{p'}(G) \geq 2$. 
 
In order to prove Theorem \ref{main}, we may thus assume that one of the following holds for the group $K/C_{K}(W)$:
\begin{enumerate}
\item [{\bf Case (1).}] $K/C_{K}(W)$ has two normal subgroups $J_1$ and $J_2$ which are minimal with respect to being noncentral in $K/C_{K}(W)$, both $J_1$ and $J_2$ are as in the first row of Table \ref{possible_exceptions_V_over_p} and $q = p \in \{  59, 61, 71, 79, 89, 101, 109\}$.

\item [{\bf Case (2).}] There is a divisor $e$ of $d/t$ such that $K/C_{K}(W)$ contains a normal subgroup $L = (K/C_{K}(W)) \cap \mathrm{GL}(d/(te),q^{e})$ of index $e$ in $K/C_{K}(W)$ acting primitively and irreducibly on a vector space $U$ over the extension field of $F$ of order $q^e$ such that 

\subitem {\bf(i).} $L$ is almost quasisimple as in the first row of Table \ref{possible_exceptions_V_over_p} with $d/(te)$ in place of $d$, $q^{e}$ in place of $q$, $U$ in place of $V$ and $L$ in place of $H$ or 

\subitem {\bf(ii).} $L$ has a unique normal subgroup $J$ which is minimal with respect to being noncentral in $L$, $J$ is of symplectic type, it acts absolutely irreducibly on $U$ and $d/(te) \in \{ 2, 3, 4 \}$. 
\end{enumerate}

\medskip

{\bf Case (1).} We prove Theorem \ref{main} in Case (1). We have $\dim_{F}(W) \geq 4$. Let $t \geq 2$. Since the number of orbits of $K$ on $W$ is at least $|W|/|K/C_{K}(W)| \geq p^{3}/120^{2}$, we find that $$k_{p}(G) \geq \frac{(p^{3}/120^{2})((p^{3}/120^{2}) +1)}{2} -1 > 104 > (p-1)/2.$$ Since $G$ is nonsolvable, $k_{p'}(G) \geq 3$ by Burnside's theorem. We may thus suppose that $t = 1$. In this case $K = H$ and $W = V$. Let $C$ be the center of $H$. The factor group $H/C$ contains $A_{5} \times A_{5}$ as a normal subgroup and is contained in $S_{5} \times S_{5}$ therefore $k(H/C) \geq 16$ and so $k_{p'}(G) = k(H) \geq |C| + 15$. Since $|H| \leq 120^{2} |C|$, we have $k_{p}(G) \geq (p^{4}-1)/(120^{2} |C|)$. If $|C| \leq 14$, then $k_{p}(G) \geq p-1$. We may thus assume that $|C| \geq 15$ and so $k_{p'}(G) \geq 30$. We have $k_{p}(G) \geq 3$ (by \cite[Theorem 1.1 and Section 2]{HungSambaleTiep}). This deals with the primes $p$ in $\{ 59, 61, 79 \}$. The integer $|C|$ must divide $p-1$. If $p=71$, then $|C|$ is divisible by $35$ and so $k_{p'}(G) \geq 50$ and $k_{p}(G) \geq 3$. Finally, if $p \in \{ 89, 101, 109 \}$, then $k_p(G) \geq 49$ and $k_{p'}(G) \geq 30$ from the above. 

\medskip

{\bf Case (2)(i).} We prove Theorem \ref{main} in Case (2)(i). We start with a lemma which holds both in Case (2)(i) and in Case (2)(ii) when $d/(te) = 2$. 

\begin{lemma}
\label{lemmae}	
Use the notation of this section. Let $G$ be a finite group and let $p$ be a prime for which Case (2)(i) or Case (2)(ii) holds, latter if $d/(te) = 2$. We have the following. 
\begin{enumerate}
\item If $t=2$ and $q > 240$, then Theorem \ref{main} holds for $G$ and $p$.

\item If $t = 3$ and $q \geq 89$, then Theorem \ref{main} holds for $G$ and $p$.

\item If $t \in \{ 4, 5 \}$ and $q \not\in \{ 59, 61 \}$, then Theorem \ref{main} holds for $G$ and $p$.

\item If $t \geq 6$ and $q \geq 600$, then Theorem \ref{main} holds for $G$ and $p$.

\item If $e \geq 2$, then Theorem \ref{main} holds for $G$ and $p$.
\end{enumerate}	
\end{lemma}	 

\begin{proof}
We give the proof in Case (2)(i). In Case (2)(ii) when $d/(te) = 2$ the proof is the same the only difference being that $120$ changes to the lower number $24$.  
	
Let $t=2$. If $|Z(B)| \geq 2e^2(p-1)$, then $k_{p'}(G) = k(H) \geq |Z(B)|/(2e^{2}) \geq p-1$ and so Theorem \ref{main} holds. Otherwise, $|H| < 2 \cdot 2e^{2}(p-1) 120^2 e^2 < (q^{4e}-1)/(p-1)$, provided that $q > 240$ or $e \geq 2$ (and $q \geq p \geq 47$), giving $k_{p}(G) \geq p-1$.   

Let $t=3$. If $|Z(B)| \geq 6e^3(p-1)$, then $k_{p'}(G) = k(H) \geq |Z(B)|/(6e^{3}) \geq p-1$ and so Theorem \ref{main} holds. Otherwise, $|H| < 6 \cdot 6e^{3}(p-1) 120^3 e^3 < (q^{6e}-1)/(p-1)$, provided that $q > 89$ or $e \geq 2$ (and $q \geq p \geq 47$), giving $k_{p}(G) \geq p-1$.

Let $t=4$. If $|Z(B)| \geq 24e^4(p-1)$, then $k_{p'}(G) = k(H) \geq |Z(B)|/(24e^{4}) \geq p-1$ and so Theorem \ref{main} holds. Otherwise, $|H| < 24 \cdot 24e^{4}(p-1) 120^4 e^4 < (q^{8e}-1)/(p-1)$, provided that $q \not\in \{ 59, 61 \}$ or $e \geq 2$ (and $q \geq p \geq 47$), giving $k_{p}(G) \geq p-1$. 

The case $t=5$ is treated as in the previous paragraph. 

Let $t \geq 6$. We have $|H/B| \leq m!^{(t-1)/(m-1)}$ and $m < {(\ln(112) + \ln p)}^{2}/4$ by the Setup paragraph above. If  $$|Z(B)| \leq \Big(\frac{ (\ln (p) + \ln (112))^{2} }{4} \Big)^{t-1}  e^{t}(p-1),$$ then Theorem \ref{main} holds. Otherwise, $$|H| <  \Big(\frac{ (\ln (p) + \ln (112))^{2} }{4} \Big)^{2t-2} e^{t}(p-1) 120^t e^t < (q^{2et}-1)/(p-1)$$ provided that $q \geq 600$ or $e \geq 2$. 
\end{proof}	

We may therefore suppose that $e = 1$. We may also suppose that $p$ is congruent to $\pm 1$ modulo $10$ by \cite[II, Hauptsatz 8.27]{huppert}.

Let $t = 2$. We may suppose that $q \leq 239$ and $|H| < 2 \cdot 2(p-1) 120^2$ by Lemma \ref{lemmae} and its proof. Since $A_5$ is a composition factor in $H$, we have $k(H) \geq k^{*}(A_5) = 4$ by \cite[Lemma 2.5]{pyber1992} and $k(H) \geq 5$ since $H \not= A_5$. Let $q = p = 239$. In this case one checks that $|H| < 2 (q^{4}-1)/(p-1)$. Let $q = 229$. We have $p-1 = 4 \cdot 3 \cdot 19$. It is sufficient to show that the number $n(H,V)-1$ of nontrivial orbits of $H$ on $V$ is larger than $57$. The inequality $|H| < q^{4}/58$ gives the result. Let $p = 211$. Then $p-1 = 2 \cdot 3 \cdot 5 \cdot 7$. It is sufficient to show that $n(H,V)-1$ is at least $70$. The inequality $|H| < q^{4}/71$ gives the result. Let $p = 199$. Then $p-1 = 2 \cdot 3^{2} \cdot 11$. It is sufficient to show that $n(H,V)-1$ is at least $66$. But $|H| < q^{4}/67$. Let $p = 191$. Then $p-1 = 2 \cdot 5 \cdot 19$. It is sufficient to show that $n(H,V)-1$ is at least $95$. But $|H| < q^{4}/96$. Let $p = 181$. Then $p-1 = 2^{2} \cdot 3^{2} \cdot 5$. It is sufficient to show that $n(H,V)-1$ is larger than $45$. But $|H| < q^{4}/46$. Let $p = 179$. Then $p-1 = 2 \cdot 89$. It is sufficient to show that $n(H,V)-1$ is at least $89$. But $|H| < q^{4}/90$. Let $p = 151$. Then $p-1 = 2 \cdot 3 \cdot 5^{2}$. It is sufficient to show that $n(H,V)-1$ is at least $50$. But $|H| < q^{4}/51$. Let $p = 149$. Then $p-1 = 2^{2} \cdot 37$. It is sufficient to show that $n(H,V)-1$ is larger than $37$. But $|H| < q^{4}/38$. Let $p = 139$. Then $p-1 = 2 \cdot 3 \cdot 23$. It is sufficient to show that $n(H,V)-1$ is at least $46$. But $|H| < (q^{4}-1)/46.96$. Let $p = 131$. Then $p-1 = 2 \cdot 5 \cdot 13$. It is sufficient to show that $n(H,V)-1$ is at least $26$. But $|H| < (q^{4}-1)/39.3$. Let $p = 109$. Then $p-1 = 2^{2} \cdot 3^{3}$. It is sufficient to show that $n(H,V)-1$ is larger than $9$. But $|H| < (q^{4}-1)/9$. Put $z = |Z(B)|$. Then $k(H) \geq z/2$ and $|H| \leq 2 120^{2}z = 28800 z$ thus $k_{p}(G) \geq (q^{4}-1)/(28800 z)$. Let $p = 101$. Then $p-1 = 2^{2} \cdot 5^{2}$. It is sufficient to show that $n(H,V)-1$ is larger than $20$. We may assume that $3613/z \leq (q^{4}-1)/(28800 z) \leq 20$, that is, $181 \leq z$. But then $k(H) \geq 181/2$. Theorem \ref{main} now follows from the assumption that $k_{p}(G) \geq 3$. Let $p=89$. It is sufficient to have $k_{p}(G) > 22$. We may assume that $2178/z \leq (q^{4}-1)/(28800 z) \leq 22$, that is, $99 \leq z$. But then $k(H) \geq 99/2$ and $k_{p}(G) \geq 3$. Let $p=79$. Here $p-1 = 78 = 3 \cdot 26$. It is sufficient to have $k_{p}(G) \geq 26$. We may assume that $(q^{4}-1)/(28800 z) \leq 26$, that is, $52.01 \leq z$. But then $k(H) \geq 27$ and $k_{p}(G) \geq 3$. Let $p = 71$. Here $p - 1 = 70 = 2 \cdot 5 \cdot 7$. It is sufficient to have $k_{p}(G) > 14$. We may assume that $(q^{4}-1)/(28800 z) \leq 14$, that is, $64 \leq z$. Thus $k(H) \geq 32$. But $k_{p}(G) \geq 3$. Let $p = 61$. Here $p-1 = 60 = 2^{2} \cdot 3 \cdot 5$. It is sufficient to get $k_{p}(G) > 12$. We may assume that $(q^{4}-1)/(28800 z) \leq 12$, that is, $41 \leq z$. We obtain $k(H) \geq 21$. But $k_{p}(G) \geq 3$. Let $p = 59$. Here $p-1 = 58 = 2 \cdot 29$. We need to show that $k_{p}(G) \geq 29$. We may assume that $(q^{4}-1)/(28800 z) \leq 29$, that is, $15 \leq z$. We also have that $z$ is even and dividing $2^{2} 29^{2}$. If $z \geq 2 \cdot 29$, then $k_{p'}(H) \geq 2 + (z-2)/2 \geq 30$. So $z$ is $2$ or $4$ which is a contradiction. 

Let $t=3$. We may suppose that $|H| < 6 \cdot 6(p-1) 120^3$ and $q=p$. Also $p$ is any of the four primes: $79$, $71$, $61$, $59$. Since $A_5$ is a composition factor of $H$ and $H \not= A_5$, we have $k(H) \geq 5$. Let $z = |Z(B)|$ as before. We have $k(H) \geq z/6$ and 
\begin{equation}
\label{displayed}	
k_{p}(G) \geq \frac{|V|-1}{|H|} \geq \frac{p^{6}-1}{6 \cdot z \cdot 120^{3}}.
\end{equation}
Let $p = 79$. Here $p-1 = 78 = 2 \cdot 3 \cdot 13$. Now $k_{p}(G) \geq 3$. It is sufficient to have $k_{p'}(G) > 26$. This is fine for $z > 156$. Let $z \leq 156$. Then $k_{p}(G) \geq 150$ by (\ref{displayed}). Let $p = 71$. Here $p-1 = 70 = 2 \cdot 5 \cdot 7$. It would be enough to have $k_{p}(G) \geq 15$. This holds for $z \geq 80$. If $z \leq 79$, then (\ref{displayed}) gives $k_{p}(G) \geq 156$. Let $p = 61$. Here $p-1 = 60 = 2^{2} \cdot 3 \cdot 5$. It is sufficient to prove $k_{p}(G) \geq 13$. This holds for $z \geq 78$. Let $z \leq 77$. Then $k_{p}(G) \geq 64$ by (\ref{displayed}). Let $p = 59$. Here $p-1 = 58 = 2 \cdot 29$. If $z \geq 6 \cdot 29 = 174$, then the result holds. Let $z \leq 173$. Now $z$ is even and divides $2^{3} 29^{3}$. If $29^{2} \mid z$, then $k_{p'}(H) \geq 2 + (z-2)/2$. Thus $z \leq 2^{2} \cdot 29$. In this case $(p^{6}-1)/(6 \cdot z \cdot 120^{3}) \geq 35$.

Let $t \in \{ 4, 5 \}$. We work with the upper bound for $|H|$ which follows from the proof of Lemma \ref{lemmae}. We have $q = p \in \{ 59, 61  \}$. Again, since $A_5$ is a composition factor of $H$ and $H \not= A_5$, we have $k(H) \geq 5$. Let $z$ be as before. 

Let $t = 4$. Let $p = 61$. Here $p-1 = 60 = 2^{2} \cdot 3 \cdot 5$. It is sufficient to show that $k_{p}(G) \geq 13$. This holds since $$k_{p}(G) \geq \frac{p^{8}-1}{|H|} \geq 26.$$ Let $p = 59$. We may assume that $z < 24 \cdot 58$. On the other hand, $z$ is even and divides $2^{4} 29^{4}$. So $z$ is $2$, $4$, $8$, $16$, $2 \cdot 29$, $4 \cdot 29$, $8 \cdot 29$, or $16 \cdot 29$ and so $$k_{p}(G) \geq \frac{p^{8}-1}{|H|} \geq \frac{p^{8}-1}{  24 \cdot 120^{4} \cdot 16 \cdot 29} \geq 63.$$   

Let $t=5$. Let $p = 61$. Here $p-1 = 60 = 2^{2} \cdot 3 \cdot 5$. It is sufficient to obtain $k_{p}(G) \geq 13$. This follows since $$k_{p}(G) \geq \frac{p^{10}-1}{|H|} \geq 35.$$ Let $p = 59$. We may assume that $z < 120 \cdot 58$. On the other hand, $z$ is even and divides $2^{5} 29^{5}$. So $z \leq 8 \cdot 29^2$ and so $$k_{p}(G) \geq \frac{p^{10}-1}{|H|} \geq \frac{p^{10}-1}{  120 \cdot  8 \cdot 29^2 \cdot 120^{5}} \geq 25.$$ Although this is not sufficient for our purpose, we are done unless $z = 8 \cdot 29^{2}$ and when $H/Z(B)$ is $S_{5} \wr S_{5}$. But this latter group has at least $(k(S_5))^{5}/120 > 140$ conjugacy classes. 

Let $t \geq 6$. Let $q = p \leq 600$. Let $m$ be the largest integer at least $4$ and less than $(\ln (p) + \ln (112))^{2}/4$. The maximum value is $30$. As before, we are done, unless $$|H| <  m!^{(2t-2)/(m-1)} (p-1) 120^t.$$ This is less than $q^{2t}/p$ unless $q \leq 232$. Let $q \leq 232$. The maximum value of $m$ is $25$. Applying the same argument, we get the result for $q \geq 198$. For $m \geq 22$ we have $k(H) \geq k^{*}(A_m) \geq \pi(m)/4 \geq 1002/4 > 196$.  We may thus assume that $m \leq 21$. For $170 \leq p \leq 197$ we get 
$$k_{p}(G) \geq \frac{|V|-1}{|H|} \geq \frac{q^{2t}-1}{21!^{(t-1)/10} (p-1) 120^t} \geq 161,$$ while $k_{p'}(G) \geq 2$. Looking at the list of numbers $k^{*}(A_{m})$ for $m \leq 21$, $k(H) \geq k^{*}(A_m) \geq 195$ for $m \geq 18$. Thus $m \leq 17$. Let $140 < p < 170$. Then $$k_{p}(G) \geq \frac{|V|-1}{|H|} \geq \frac{q^{2t}-1}{17!^{(t-1)/8} (p-1) 120^t} \geq 109,$$ while $k_{p'}(G) \geq 2$. But then $m \leq 16$. Let $130 < p < 140$. Then $$k_{p}(G) \geq \frac{|V|-1}{|H|} \geq \frac{q^{2t}-1}{16!^{(t-1)/7.5}(p-1) 120^t} > 79,$$ while $k_{p'}(G) \geq 2$. We have $k_{p}(G) \geq 3$. We are done for $m \geq 14$. Thus $m \leq 13$. Let $110 < p < 130$. Then $$k_{p}(G) \geq \frac{|V|-1}{|H|} \geq \frac{q^{2t}-1}{13!^{(t-1)/6} (p-1) 120^t} > 66,$$ while $k_{p'}(G) \geq 2$. Let $101 \leq p < 110$. Let $m=13$. Then $$k_{p}(G) \geq \frac{|V|-1}{|H|} \geq \frac{q^{2t}-1}{13!^{(t-1)/6}(p-1) 120^t} > 25,$$ while $k_{p'}(G) \geq k^{*}(A_{m}) = 52$. Here only the prime $107$ remains as $p-1 = 2 \cdot 53$. Since $H \not= A_m$, we have $k(H) \geq k^{*}(A_{m}) + 1$. The range for $p$ remains. Let $m=12$. Then $$k_{p}(G) \geq \frac{|V|-1}{|H|} \geq \frac{q^{2t}-1}{12!^{(t-1)/5.5} (p-1) 120^t } > 48,$$ while $k_{p'}(G) \geq k^{*}(A_{m}) = 40$. For $p = 107$ the $48$ on the right-hand side of the previous inequality changes to $91$. Let $m \leq 11$. Then $$k_{p}(G) \geq \frac{|V|-1}{|H|} \geq \frac{q^{2t}-1}{11!^{(t-1)/5} (p-1) 120^t} > 94,$$ while $k_{p'}(G) \geq 2$. Let $79 \leq p \leq 97$. We have $m \leq 12$. Let $m=12$. Then $$k_{p}(G) \geq \frac{|V|-1}{|H|} \geq \frac{q^{2t}-1}{12!^{(t-1)/5.5}(p-1) 120^t} > 3,$$ while $k_{p'}(G) \geq 40$ but surely there is one more class (for the prime $p = 83$ which we will only consider in Case (2)(ii)). Let $m=11$. Then $$k_{p}(G) \geq \frac{|V|-1}{|H|} \geq \frac{q^{2t}-1}{11!^{(t-1)/5} (p-1) 120^t} > 6,$$ while $k_{p'}(G) \geq 29$. This deals with all primes $p$ except $p = 83$, which is not congruent to $\pm 1$ modulo $10$ and therefore it is not considered here, but in Case (2)(ii) we may replace $120$ by $24$ and the relevant case will follow. Let $m=10$. Then $$k_{p}(G) \geq \frac{|V|-1}{|H|} \geq \frac{q^{2t}-1}{10!^{(t-1)/4.5} (p-1) 120^t} > 13,$$ while $k_{p'}(G) \geq 22$. Again, we have the same issue with $83$. Let $m = 9$. In this case $k_{p'}(G) \geq 16$, and we have the same issue with $83$. Let $m \leq 8$. Then $$k_{p}(G) \geq \frac{|V|-1}{|H|} \geq \frac{q^{2t}-1}{8!^{(t-1)/3.5} (p-1) 120^t} > 66.$$ The remaining primes are $47$, $53$ (in Case (2)(ii)), $59$, $61$, $67$, $71$, $73$ (in Case (2)(ii)). The previous computation shows that we may assume that $m \geq 9$. If $m \geq 12$, then $k_{p'}(G) \geq 40$ and we are done. Thus $m \in \{ 9, 10, 11 \}$. Observe that $t \geq m$. Then $$k_{p}(G) \geq \frac{|V|-1}{|H|} \geq \frac{q^{2t}-1}{11! (p-1) 120^t } > 132$$ for $t \geq 9$.  
 
\medskip 
 
{\bf Case (2)(ii).} Finally we prove Theorem \ref{main} in Case (2)(ii). 

Let $d/(te)=2$. The case $t=1$ was treated in Section 5. For $t \geq 2$ we apply Lemma \ref{lemmae} and we follow the proof in Case (2)(i) with $120$ replaced by $24$. 

Let $d/(te)=3$. We have $|L| \leq 3^2  |C| |\mathrm{Sp}(2,3)| = 216 |C|$ where $C = Z(L)$. This is less than $|W|/q^{e} = q^{2e}$ for $q^{e} \geq 217$ and Theorem \ref{main} follows in this case for all $t$. We may thus suppose that $q^{e} < 217$. Since $p \geq 47$, this implies that $q = p$ and $e = 1$. We get $p \leq 211$. Let $t=1$. Now $|C|$ divides $p-1$ and $k_{p'}(G) = k(H) \geq |C|$. On the other hand, $k_{p}(G) \geq (|W|-1) / (216|C|) > (p-1)/|C|$. Now $(|W|-1) / (216|C|) > 10$ and so $k_{p}(G) \geq (12 \cdot 13)/2 - 1 = 77$ for $t=2$ and $k_{p}(G) \geq (12 \cdot 13 \cdot 14)/6 - 1 = 363$ for $t \geq 3$. The case $t=2$ remains. If $p > 70$, then $(|W|-1) / (216|C|) > 23$ and so $k_{p}(G) \geq (24 \cdot 25)/2 - 1 \geq 299$. Thus $p \leq 67$. Since $3 \mid p-1$, the remaining primes are $61$ and $67$. We have $(|W|-1) / (216|C|) > 17$ and so $k_{p}(G) \geq (18 \cdot 19)/2 - 1 = 170$. 

Let $d/(te)=4$. We have $|L| \leq 2^4 |C| |\mathrm{Sp}(4,2)| = 11520 |C|$ where $C = Z(L)$. This is less than $|W|/q^{e} = q^{3e}$ for $q^{e} \geq 109$ and Theorem \ref{main} holds in this case for all $t$. We may thus suppose that $q^{e} < 109$. Since $p \geq 47$, this implies that $q = p$ and $e = 1$. Let $t=1$. We have $k_{p'}(G) = k(H) \geq |C|$ and $|C|$ is a divisor of $p-1$. It suffices to show that $k_{p}(G) \geq (q^{4}-1)/(11520 |C|) > (p-1)/|C|$, that is, that $(q^{4}-1)/11520 > p-1$. This is the case for $p \geq 47$. Let $t \geq 2$. We have $n(L, W) - 1 \geq (q^{4}-1)/(11520 (q-1)) > 9$ and so $k_{p}(G) \geq (10 \cdot 11)/2 - 1 = 54 > (p-1)/2$ (for $p < 109$).    

\bigskip
 
\begin{center}\bf  Acknowledgement
\end{center}

\bigskip
The authors thank Alexander Moret\'o for a helpful comment on an earlier version of this paper. Part of this work was done while the second and fourth authors visited the third author at the Alfr\'ed R\'enyi Institute of Mathematics in April 2024. They would like to thank the Institute for its hospitality. The second author was on sabbatical leave from Texas State University. The fourth author was supported by the Babe\c s-Bolyai University through grant number SRG-UBB 32910. Also, some part of this work  was done while the first author visited the second author as a Research Fellow, supported  by the Scientific and Technological Research Council of T\"urkiye, at Texas State University. She  would like to thank the Department of Mathematics at Texas State University for its hospitality, and T\"UB\.ITAK for granting her the research fellowship.

\end{document}